\documentclass[a4paper,reqno]{amsart}

\usepackage{amssymb,bbm,amsthm}
\usepackage[british]{babel}
\usepackage{enumerate}
\usepackage[usenames,dvipsnames]{xcolor}
\usepackage[all,cmtip]{xy}
\usepackage{verbatim}
\usepackage[normalem]{ulem}

\newtheorem {theorem}{Theorem}
\newtheorem*{theorem*}{Theorem}
\newtheorem*{theoremA*}{Theorem A}
\newtheorem*{theoremB*}{Theorem B}
\numberwithin{theorem}{section}
\numberwithin{equation}{section}
\newtheorem {corollary}[theorem]{Corollary}
\newtheorem*{corollary*}{Corollary}

\newtheorem {lemma}[theorem]{Lemma}
\newtheorem {proposition}[theorem]{Proposition}

\newtheorem {conjecture}[theorem]{Conjecture}
\theoremstyle{definition}
\newtheorem {remark}[theorem]{Remark}
\newtheorem {definition}[theorem]{Definition}
\newtheorem*{definition*}{Definition}

\newcommand{\vm}[1] { {\color{OliveGreen} \textsf{[[VM: #1]]} } }

\newcommand{\dl}[1] { {\color{blue}       \textsf{[[DL: #1]]} } }
\newcommand{\gf}[1] { {\color{purple}     \textsf{[[GF: #1]]} } }

\newcommand{\Spec}     {\operatorname{Spec}}
\newcommand{\Funct}    {\operatorname{Funct}}
\newcommand{\Ind}      {\operatorname{Ind}}
\newcommand{\Res}      {\operatorname{Res}}

\newcommand{\Hom}      {\operatorname{Hom}}

\newcommand{\End}      {\operatorname{End}}

\newcommand{\Op}       {\operatorname{Op}}

\newcommand{\het}      {\operatorname{ht}}
\newcommand{\QCoh}      {\operatorname{QCoh}}

\newcommand{\Id}            {{\mathrm{Id}}}
 
\newcommand{\statovuotoF}   {|0\rangle_\Fock}
\newcommand{\statovuotoFd}  {|0\rangle_{\Fock_2}}
\newcommand{\statovuotoV}   {|0\rangle_\mV}
\newcommand{\statovuoto}    {|0\rangle}
\newcommand{\semiinfz}      {\Psi^0}
\newcommand{\semiinf}       {\Psi}

\newcommand{\hUCl}[1]   {\hU_{#1} \hat \otimes \Cl_{#1}}
\newcommand{\hUClt}   {\hU_{t} \hat \otimes_Q \Cl_{t}}
\newcommand{\hUCls}   {\hU_{s} \hat \otimes_Q \Cl_{s}}
\newcommand{\hUClu}   {\hU_{1} \hat \otimes \Cl_{1}}
\newcommand{\hUCld}   {\hU_{2} \hat \otimes_A \Cl_{2}}
\newcommand{\hUClts}   {\hU_{t,s} \hat \otimes_Q \Cl_{t,s}}
\newcommand{\hU}        {{\hat U}}
\newcommand{\hg}        {{\hat \gog}}

\newcommand{\uno}       {\mathbbm{1}}
\newcommand{\Weyl}      {\mV^{\grl,\mu}_2}
\newcommand{\tWeyl}     {\widetilde \mV^{\grl,\mu}_2}

\newcommand{\Sp}       {\operatorname{Sp}}

\newcommand{\hgog}   {{\hat\gog}}

\newcommand{\limpro} {\varprojlim}

\newcommand{\Cl}{{\calC\ell}}

\newcommand{\Fock}{{\Lambda}}
\newcommand{\field}{{\calF}}
\newcommand{\bigrado}{\operatorname{bideg}}
\newcommand{\cha}{\operatorname{ch}}

\newcommand{\pr}{\operatorname{pr}}
\newcommand{\std}{\operatorname{std}}

 \newcommand{\mC}{\mathbb C} 
 \newcommand{\mF}{\mathbb F}

  \newcommand{\mV}{\mathbb V}
\newcommand{\mW}{\mathbb W}  
\newcommand{\mZ}{\mathbb Z}

 \newcommand{\calC}{\mathcal C} 
 \newcommand{\calF}{\mathcal F} \newcommand{\calG}{\mathcal G}
\newcommand{\calH}{\mathcal H}  
\newcommand{\calK}{\mathcal K}  
  
 \newcommand{\calR}{\mathcal R} 
  \newcommand{\calV}{\mathcal V}

\newcommand{\goa}  {\mathfrak a}
\newcommand{\gob}{\mathfrak b}  
  \newcommand{\gog}{\mathfrak g}
  
 \newcommand{\gol}{\mathfrak l} 
\newcommand{\gon}{\mathfrak n}  
  \newcommand{\gos}{\mathfrak s}
\newcommand{\got}{\mathfrak t}

\newcommand{\gra}{\alpha} \newcommand{\grb}{\beta}       \newcommand{\grg}{\gamma}
  
 \newcommand{\grl}{\lambda}     
\newcommand{\grf}{\varphi}      
\newcommand{\grG}{\Gamma}

\newcommand{\ra}       {\rightarrow}
\newcommand{\lra}      {\longrightarrow}

\renewcommand{\geq}    {\geqslant}
\renewcommand{\leq}    {\leqslant}
\newcommand{\senza}    {\smallsetminus}

\newcommand{\DX}{\mathcal{D}_{X}}

\newcommand{\gc}{\widehat{\mathfrak{g}}_{\operatorname{crit}}}

\newcommand{\gcm}{\gc\text{-mod}}

\newcommand{\Aff}{^f\text{Sch}_{\operatorname{aff}/k}}
\newcommand{\D}{\mathcal{D}}
\newcommand{\A}{\mathcal{A}}
\newcommand{\CC}{\mathcal{C}}
\newcommand{\QC}{\operatorname{QCoh}}
\newcommand{\Y}{\mathcal{Y}}
\newcommand{\ZZ}{\mathfrak{Z}_{\operatorname{crit}}}

\newcommand{\CACCA}{\text{Op}^{\circ}_{\check{\gog},X}}
\newcommand{\ACM}{\A_{\operatorname{crit}}\operatorname{-mod}}

\newcommand{\OPDIc}{\text{Op}^{\circ}_{\check{\gog},I}}
\newcommand{\OPDUGc}{\mathrm{Op}_{\check{\gog}}^{\mathrm{unr}}}

\newcommand{\OPDUXc}{\text{Op}_{\check{\gog}, X}^{\operatorname{unr}}}

\newcommand{\OPPDI}{\text{Op}^{\circ}_{\gog,I}}

\allowdisplaybreaks

\begin{document}

\author{Giorgia Fortuna, Davide Lombardo, \\ Andrea Maffei, Valerio Melani}
\renewcommand{\shortauthors}{Fortuna, Lombardo, Maffei, Melani}

\title[Semi-infinite cohomology of Weyl modules]{The semi-infinite cohomology of Weyl modules with two singular points}

\maketitle

\begin{center}
	{\it A Claudio Procesi, con ammirazione.}
\end{center}

\bigskip

\begin{flushright}``a volte due punti sono pi\`u vicini di quanto non sembri,\\ ma per unirli ci vuole un'idea. Lui era la persona giusta.''\\ Nonmaterial lifeform, di F. B. Amadou, Urania. 
\end{flushright}

\begin{abstract}
In their study of spherical representations of an affine Lie algebra at the critical level and of unramified opers, Frenkel and Gaitsgory introduced what they called the \textit{Weyl module} $\mV^\grl$ corresponding to a dominant weight $\grl$. This object plays an important role in the theory. In \cite{FLMM}, we introduced a possible analogue $\mV_2^{\grl,\mu}$ of the Weyl module in the setting of opers with two singular points, and in the case of $\gos\gol(2)$ we proved that it has the `correct' endomorphism ring. In this paper, we compute the semi-infinite cohomology of $\mV_2^{\grl,\mu}$ and we show that it does not share some of the properties of the semi-infinite cohomology of the Weyl module of Frenkel and Gaitsgory. For this reason, we introduce a new module $\tilde \mV_2^{\grl,\mu}$ which, in the case of $\gos\gol(2)$, enjoys all the expected properties of a Weyl module.
\end{abstract}

\section{Introduction}

Let $\gog$ be a complex simple Lie algebra and let $\hgog$ be its affinization. 
Choose a Borel subalgebra and a maximal toral subalgebra, and let $G$ be a simply connected algebraic group with Lie algebra equal to $\gog$. As a particular case of a more general conjecture, Frenkel and Gaitsgory proved in \cite{FG7} that the semi-infinite cohomology gives an isomorphism between the category $\hgog_{crit}\text{-}\mathrm{mod}^{JG}$ of spherical representations of $\hgog$ at the critical level (that is, representations of $\hgog$ at the critical level with a compatible action of $JG=G(\mC[[t]])$) and the category of quasi-coherent sheaves on the space of unramified opers $\Op_1^{\mathrm{unr}}$ over $\gog^L$, the Langlands dual of $\gog$. As they explain, the space of unramified opers is the disjoint union of its connected components $\Op_1^{\grl,\mathrm{unr}}$, and the category of spherical representations is the product of certain subcategories $\hgog_{crit}\text{-}\mathrm{mod}^{JG,\grl}$, where in both cases $\grl$ ranges over all dominant weights of $G$. The equivalence given by semi-infinite cohomology specialises to an equivalence between $\hgog_{crit}\text{-}\mathrm{mod}^{JG,\grl}$ and the category of quasi-coherent sheaves over $\Op_1^{\grl,\mathrm{unr}}$. The space $\Op_1^{\grl,\mathrm{unr}}$ is a non-reduced indscheme, and its reduced version, denoted by $\Op_1^\grl$, is an affine scheme. In this paper we will denote by $Z_1^\grl$ its coordinate ring.

In this theory, an important role is played by the Weyl module $\mV_1^\grl$. This module enjoys the following fundamental properties:
$$ \End_{\hgog}(\mV_1^\grl)\simeq Z_1^\grl \quad\text{ and } \quad \semiinfz(\mV_1^\grl)\simeq Z_1^\grl, $$ 
where $\semiinf^n$ is the $n$-th semi-infinite cohomology group. Moreover the semi-infinite cohomology groups  $\semiinf^n(\mV_1^\grl)$ are trivial for $n\neq 0$. 

Dennis Gaitsgory suggested to Giorgia Fortuna to study the space of unramified opers and spherical representations in a more general context, see \cite{GiorgiaPhD}; in fact, the definition of unramified opers as well as the definition of spherical representations can be generalized in the presence of more than one singularity, raising the question on whether or not certain statements remain true and what happens when these singularities collide. 

In \cite{FLMM} we took some steps in this direction, by studying the case of $\gos\gol(2)$. In particular, we introduced a version of the Weyl module $\mV^{\grl,\mu}_2$ of critical level of the affine Lie algebra with two singularities $\hgog_2$. Thinking of $t$ as a coordinate near the first singularity and $s$ as a coordinate near the second singularity, this is the version of the affine Lie algebra over the ring $A=\mC[[a]]$, where $a = (t-s)$. As an $A$ module is equal to $K_2\otimes_\mC\oplus AC_2$
where $K_2=\mC[[a,t]][1/t(t-a)]$ and $C_2$ is a central element (see \cite{FLMM}, Section 3.3 for the complete definition). 

\noindent We also introduced reduced scheme over $A$ of unramified opers $\Op^{\grl,\mu}_2$ which generalize the schemes $\Op^{\grl}_1$. Both objects depend on two integral dominant weights $\grl$, $\mu$ of $G$, and we proved that 
$$ \End_{\hgog_2}(\mV_2^{\grl,\mu})\simeq Z_2^{\grl,\mu}, $$
where $Z_2^{\grl,\mu}$ is the coordinate ring of $\Op_2^{\grl,\mu}$. 

In this article we study the semi-infinite cohomology of $\Weyl$ and its relation with the ring $Z_2^{\grl,\mu}$ in order to understand how the equivalence $\semiinfz(\mV_1^\grl)\simeq Z_1^\grl$ generalizes.

\noindent This is done in Section \ref{sec:calcolocomologia}, where we compute the cohomology of $\Weyl$; in Section \ref{sec:azionecentro} we study the action of $Z_2$, the center of a completion $\hU_2$ of the enveloping algebra of $\hgog_2$ at the critical level on this module (see Section \ref{sez:defU}).  \\
In particular, we prove that the specialisation at $a=0$ and the localization at $a\neq 0$ of the semi-infinite cohomology of $\Weyl$ are isomorphic to the specialisation and localization of $Z_2^{\grl,\mu}$, respectively. However, in contrast to our intuition, we also show the following result which says that $\semiinf^0(\Weyl)$ doesn't exactly generalize the equivalence $\semiinfz(\mV_1^\grl)\simeq Z_1^\grl$ as expected:

\begin{theoremA*}[Theorem \ref{teo:comologiaWeyl2} and Proposition \ref{prop:noniso}]
	We have $\semiinf^n(\Weyl)=0$ for $n\neq 0$. Moreover, $\semiinfz(\Weyl)$
	is not isomorphic to $Z^{\grl,\mu}_2$ as a $Z_2$-module. 
\end{theoremA*}

For this computation, we rely on the formalism introduced by Casarin in \cite{Casarin1}, which makes it possible to use vertex algebras also in the context of opers with two singularities. Once this formalism is in place, for the computation of the semi-infinite cohomology we can follow closely the approach taken by Frenkel and Ben Zvi in \cite[Chapter 15]{FB} for the case of one singularity.  

In the last section, we restrict our attention to the Lie algebra $\gos\gol(2)$ and introduce a submodule 
$\tWeyl$ of $\Weyl$, which is generated by the highest weight vector. We prove that this module is the correct one to consider, in the sense that it has the expected cohomology groups and endomorphism ring, as the following result shows.

\begin{theoremB*}[Proposition \ref{prp:comologiatWeyl1}, Theorem \ref{teo:comologiatWeyl2} and Proposition \ref{prp:endomorphismVtilde}] If $\gog=\gos\gol(2)$ then we have $\semiinf^n(\tWeyl)=0$ for $n\neq 0$. Moreover, we have 
	$$ \End_{\hgog_2}(\tWeyl)\simeq Z_2^{\grl,\mu} \quad\text{ and } \quad \semiinfz(\tWeyl)\simeq Z_2^{\grl,\mu}. $$ 
\end{theoremB*}

We now briefly explain the connection between these results and Conjecture 3.6.1 in Fortuna's Thesis \cite{GiorgiaPhD}. As a particular case the conjecture predicts an equivalence between quasi-coherent sheaves over the space of unramified opers with two singularities and the category of spherical representations over $\hgog_2$: that is the space of smooth representations of $\hgog_2$ with a compatible action of $J_2G=G(\mC[[a,t]])$. 
	
The conjecture stated in \cite{GiorgiaPhD} predicts an equivalence of similar categories not only in the presence of two singularities but in the presence of $n$-possible singularities. In particular for any finite set with $n$ elements $I$ we can define the space of opers on the formal disc with $n$-singularities $\Op_I$ and the subspace of unramified opers $\Op^{\mathrm{unr}}_I$ (see Section 3.5 in \cite{GiorgiaPhD}). These are spaces over the product of $n$-copies of the formal disc. These are easily seen to be factorization spaces, which means that this spaces specialise nicely when restricted along or outside the diagonals of this product (see Section 3.1.5 in \cite{GiorgiaPhD}). There are not substantial differences between the treatment we do here or in \cite{FLMM} of $\Op_2$ and the general case. The only minor difference is that we fix a singularity to be $0$. These spaces are indschemes, and so we can define the categories $\QCoh(\Op_I)$, and $\QCoh(\Op^{\mathrm{unr}}_I)$ of quasi-coherent sheaves on $\Op_I$ and $\Op_I^{ur}$ (see Section 3.5.3 in \cite{GiorgiaPhD} for the actual definition), and the nice factorization properties which make them factorization categories (see Section 3.1.2 in \cite{GiorgiaPhD}).  

Similarly, for a finite set $I$ we can define a Lie algebra $\hgog_I$ and study its smooth representations at the critical level. The objects constructed in this way live also on the product of $n$ copies of the formal disc, and they also have nice factorization properties, in particular the collection of (completions of the) enveloping algebras specialized at the critical level $\hU_I$ of the algebras $\hgog_I$, is what is called a factorization algebra (see Section 3.1.3 in \cite{GiorgiaPhD}). As a conseguence the collection of the categories of smooth representations at the critical of the Lie algebras $\hgog_I$, denoted by $\hgog_{I,crit}\text{-}\mathrm{mod}$ and their subcategories of spherical representations $\hgog_{I,crit}\text{-}\mathrm{mod}^{JG}$ can be organized also in a factorization category. The semi-infinite cohomology can be defined also in this generality and defines a functor $$\Psi_I:\hgog_{I,crit}\text{-}\mathrm{mod}\lra D(\QCoh(\Op_I))$$ compatible with the factorization properties. 
While in Fortuna's thesis all these constructions are obtained somehow for free using the language of chiral algebras (see Section 3.1.6 in \cite{GiorgiaPhD}), in this paper we use the language of vertex algebras and the formalism introduced by Casarin \cite{Casarin1}. Let us notice that, from this point of view, there are no differences in treating the case with two singular points and the case with an arbitrary finite number of singular points. For example, the proof of Theorem A above can be repeated verbatim in the case of $n$ singular points. More generally we believe that all the technical difficulties in the study of this problem already appear in the case of two singularities. 

It is easy to see from the factorization properties and the analogous statement for the case of one singularty by Frenkel and Gaitsgory (see \cite{FG6}) that the semi-infinite cohomology of a $\hgog_I$-spherical module   is supported on $\Op^{\mathrm{unr}}_I$. Hence semi-infinite cohomology restricts to a functor $\Psi_I:\hgog_{I,crit}\text{-}\mathrm{mod}^{JG}\lra D(\QCoh(\Op^{\mathrm{unr}}_I))$. Conjecture 3.6.1 in \cite{GiorgiaPhD} states that this functor is exact and that
$$ \Psi_I^0:\hgog_{I,crit}\text{-}\mathrm{mod}^{JG}\lra \QCoh(\Op^{\mathrm{unr}}_I) $$
is an equivalence of categories. In fact, it can be seen that the first part of Theorem A implies that $\Psi_I$ is exact. Moreover, in the case of $\gog=\gos\gol(2)$, Theorem B yields that the restriction of $\Psi^0_I$ to modules with reduced support is an equivalence. The details will be given in a forthcoming paper. 

\medskip

The paper is organized as follows. In the first section we recall some definitions from \cite{FLMM}. In Section \ref{sect:Vertex algebras} we recall the formalism introduced by Casarin \cite{Casarin1} and we use it to define semi-infinite cohomology and prove some of its basic properties. In Sections 3 and 4 we compute the semi-infinite cohomology of $\Weyl$ and in Section 5 we compute the semi-infinite cohomology of $\tWeyl$. 

We thank Luca Casarin for many useful discussions and in particular for explaining to us the formalism introduced in \cite{Casarin1}. It seems to us that Casarin's approach provides a natural framework to treat questions concerning opers with several singularities, making the theory much more transparent than it was in \cite{FLMM}. In particular, the results of \cite{Casarin1} allowed us to streamline several arguments and calculations which would have been quite hard to carry out using the direct approach of \cite{FLMM}.

\section{Basic constructions}
In this section we recall some basic constructions from \cite{FLMM}, to which we refer for further details, and we introduce the notion of semi-infinite cohomology in the context of affine Lie algebras with more than one singular point. 

\subsection{Rings} We follow \cite[Section 1]{FLMM}, to which the reader is referred for more details. We introduce the rings
$$
A=\mC[[a]],\qquad 
Q=\mC((a)),\qquad
R_2=\mC[[t,s]],\qquad
K_2=\mC[[t,s]][1/ts],
$$
where $a=t-s$. Recall that we have expansion maps (given by suitable natural inclusions) and a specialisation map (which sends $a$ to $0$ and $t,s$ to $t$, see Section 1.1 in \cite{FLMM})
$$
E_t:K_2[a^{-1}]\lra Q((t)), \qquad E_s:K_2[a^{-1}]\lra Q((s)),\qquad \Sp:K_2\lra \mC((t)).
$$
We also write $E=E_t\times E_s :K_2[a^{-1}]\lra Q((t))\times Q((s))$. Recall from \cite[Section 1.1]{FLMM} that $\Sp$ induces an isomorphism $K_2/(a)\simeq \mC((t))$. These rings have natural topologies: with respect to these, the image of $E$ is dense, and $E(R_2[a^{-1}])$ is dense in $Q((t)) \times Q((s))$.

These rings are also equipped with residue maps
$$
\Res_2:K_2\ra A\quad \Res_1:\mC((t))\ra \mC,
\quad \Res_t:Q((t))\ra Q,
\quad \Res_s:Q((s))\ra Q,
$$
which behave nicely with respect to specialisation and expansion (see \cite[Section 1.2]{FLMM}).
Finally, we recall Lemma 1.10 in \cite{FLMM}. 

\begin{lemma}[\cite{FLMM}, Lemma 1.10]\label{lem:isoMN}
	Let $M, N$ be two $A$-modules and $\grf:M\lra N$ be a morphism of $A$-modules. Then
	\begin{enumerate}[a)]
		\item if $M$ is flat and $\grf_a:M[a^{-1}]\lra N[a^{-1}]$ is injective, then $\grf$ is injective.
		\item if $N$ is flat, $\grf_a:M[a^{-1}]\lra N[a^{-1}]$ is surjective, and $\overline{\grf} :M/aM\lra N/aN$ is injective, then $\grf$ is surjective.
	\end{enumerate}
	In particular, if $M$ and $N$ are flat, $\grf_a:M[a^{-1}]\lra N[a^{-1}]$ is an isomorphism, and $\overline{\grf}:M/aM\lra N/aN$ is injective, then $\grf$ is an isomorphism.
\end{lemma}

\subsection{Affine Lie algebras and completion of the enveloping algebra}\label{sez:defU}
We follow \cite[Section 3]{FLMM}. Let $\gog$ be a finite-dimensional Lie algebra over the complex numbers and denote by $\kappa$ the Killing form of $\gog$. 
Recall from \cite[Sections 3.1 and 3.3]{FLMM} that for each of the rings of the previous section we introduce an affine Lie algebra: $\hgog_1$ is the usual affine Lie algebra (we take for convenience the version defined by Laurent polynomial and not Laurent series), 
$\hgog_t$ and $\hgog_s$ are also versions of the  usual affine Lie algebra, 
while $\hgog_2$ is an $A$-Lie algebra having as underlying $A$-module the space
$$
\hgog_2=\mC[t,s][1/ts]\otimes_\mC \gog \oplus A \, C_2.
$$
We also introduce the Lie algebra $\hgog_{t,s}=\hgog_t\oplus \hgog_s/(C_t-C_s)$ (see
\cite[Section 3.3]{FLMM}).

For each of these Lie algebras, we introduce the corresponding universal enveloping algebra, which we suitably complete and then specialize at the critical level by imposing that the central element acts as $-1/2$ (see Sections 3.1 and 3.3 in \cite{FLMM}). In particular 
$$
\hU_2=\limpro_n \frac{U(\hgog_2)}{(C_2=-1/2, \; t^ns^n\mC[t,s]\otimes \gog)_{\operatorname{left.id.}}}
$$
Recall from \cite[Section 3.4]{FLMM} that the expansion maps and the specialisation maps induce morphisms at the level of Lie algebras. In particular, the specialisation map $\Sp:\hU_2\lra \hU_1$ induces an isomorphism between $\hU_2/a\hU_2$ and $\hU_1$, while the expansion map induces a morphism $E:\hU_2[a^{-1}]\lra \hU_{t,s}$ which is injective and has dense image. 

Moreover, the natural inclusions $\hgog_t\hookrightarrow \hgog_{t,s}$ and 
$\hgog_s\hookrightarrow \hgog_{t,s}$ induce a morphism 
$$
\hU_t\otimes \hU_s\lra \hU_{t,s}
$$
which is also injective and with dense image (see \cite[Section 3.3]{FLMM}).

\subsection{Weyl modules} We follow \cite[Section 6]{FLMM}. We choose a Borel subalgebra and a maximal toral subalgebra of $\gog$, which we denote by $\gob$ and $\got$ respectively. This data induces a choice of weights, integral weights and dominant weights. For every integral dominant weight $\grl$, \cite{FG6} introduced the Weyl module $\mV^\grl_1$ over the affine Lie algebra $\hgog_1$. The representation $\mV=\mV_1^0$, which has a structure of vertex algebra, will play a particularly important role for us. This vertex algebra enjoys the following universal property.
\begin{lemma}\label{lemma: universal property Weyl module superalgebra}
Let $U$ be a vertex algebra such that there exists a linear map $x\mapsto u_x$ from $\gog$ to $U$ such that
$$
(u_x)_{(0)}(u_y)=u_{[x,y]}\qquad 
(u_x)_{(1)}(u_y)=-\frac 12 \kappa(x,y) \statovuoto_U \qquad 
(u_x)_{(n)}(u_y)=0
$$
for all $n\geq 2$. There exists a unique morphism of vertex algebras $\gra:\mV\ra U$ such that 
$\gra(xt^{-1}\statovuotoV)=u_x$ for all $x\in \gog$.
\end{lemma}

Weyl modules $\mV^\grl_t$ and $\mV^\grl_s$ can also be defined for the Lie algebras $\hgog_t$ and $\hgog_s$, without any significant change from \cite{FG6}. In \cite{FLMM}, we introduced a generalization of these modules. Given two dominant weights $\grl,\mu$, we consider the irreducible representations $V^\grl$ and $V^\mu$ of the Lie algebra $\gog$ having highest weights $\grl,\mu$, respectively. In \cite[Definition 6.2]{FLMM}, given two dominant integral weights $\lambda,\mu$ we introduced the module
\[
\mV_2^{\lambda, \mu} = \Ind_{\hgog_2^+}^{\hgog_2} \left(A\otimes_\mC 
V^\lambda \otimes_\mC V^\mu \right),
\]
where $\hgog_2^+= \mC[t,s]\otimes \gog\oplus A\, C_2$ acts on $A\otimes_\mC 
V^\lambda \otimes_\mC V^\mu$ as $$f(t,s)x \cdot (p(a)\otimes u \otimes v) = f(0,-a)p(a)\otimes xu \otimes v + f(a,0)p(a)\otimes   u 
\otimes xv,$$ while $ C_2 $ acts as $-1/2$. In \cite{FLMM} we called this object \textit{the Weyl module of weights $(\lambda, \mu)$}, although, as we will see, it does not have the same properties as its 1-singularity analogue.

We also define
$$ \mW_1^{\grl,\mu}=\Ind_{\hgog_1^+}^{\hgog_1} \left(V^\lambda \otimes_\mC V^\mu \right), $$
where $\hgog_1^+= \mC[t]\otimes \gog\oplus \mC\, C_1$ acts on 
$V^\lambda \otimes_\mC V^\mu$ as $f(t)x \cdot ( u \otimes v) = f(0)x\cdot(u \otimes v)$ and $C_1$ acts as $-1/2$.

The specialisation and expansion maps are defined also for Weyl modules, and induce the following isomorphisms \cite[Lemma 6.3]{FLMM}:
\begin{equation}\label{eq:SEWeyl}
\frac{\mV_2^{\grl,\mu}}{a\mV_2^{\grl,\mu}}\simeq\mW_1^{\grl,\mu},\qquad 
\mV_2^{\grl,\mu}[a^{-1}]\simeq \mV^\grl_t\otimes_Q \mV_s^\mu.
\end{equation}

\subsection{Clifford algebra}\label{sec:Cliff} We now define the Clifford algebra with two singularities, generalizing the construction of the classical case (see for example \cite[Chapter 15]{FB}). 
Let $\gon_+$ be the nilpotent radical of $\gob$ and set 
$$
 X_2 = K_2\otimes_\mC \gon_+ \oplus K_2 \otimes_\mC \gon^*_+.
$$
We equip $X_2$ with the unique $A$-bilinear form such that $K_2\otimes_\mC \gon_+$ and $K_2 \otimes_\mC \gon^*_+$ are isotropic subspaces and
$$
(f\otimes x ; g\otimes \grf)=\Res_2(fg)\,  \grf(x)
$$
for all $f,g\in K_2$, $x\in \gon_+$ and $\grf\in \gon^*_+$. We denote by $\Cl_2$ the associated Clifford algebra over $A$. 

There are obvious variants of the same construction where we replace $K_2$ with the ring 
$\mC[t^{\pm 1}]$ or one of the rings $Q[t^{\pm 1}]$, $Q[s^{\pm 1}]$, $Q[t^{\pm 1}]\times Q[s^{\pm 1}]$. We obtain Clifford algebras that we denote by $\Cl_1,\Cl_t,\Cl_s,\Cl_{t,s}$. 
The algebra $\Cl_U$ in \cite[Section 15.1.1]{FB} is a completion of $\Cl_1$.

These Clifford algebras have a natural grading called the \textit{charge} and denoted by $\cha$. It can be defined as follows: the elements of the base ring have charge 0, while for $\psi\in \gon$ and $\psi^*\in \gon^*$ we have 
\begin{equation}\label{eq:charge}
 \cha \psi =-1, \qquad \cha \psi^*=1.
\end{equation}
The relations defining each Clifford algebra are homogeneous, hence the charge induces a well-defined grading on the Clifford algebra.

   We now introduce completions of the tensor product $\hU_2\otimes_A \Cl_2$. We define
$$
\hU_{2} \hat \otimes_A \Cl_{2} = \limpro_{n}\frac{\hU_2\otimes_A \Cl_2}{\left((ts)^n R_2\gog\otimes 1,\; 1\otimes(ts)^nR_2\gon_+,\;1\otimes(ts)^nR_2\gon^*_+\right)_{\textup{left\,ideal}}} 
$$
and we notice that, as in the case of the algebra $\hU_2$, this $A$-module has a natural structure of $A$-algebra. We introduce the completed Clifford algebras $\hUClu$, $\hUClt$, $\hUCls$, and $\hUClts$. The specialisation and expansion map determine morphisms
$$
\Sp:\hUCld\lra \hUClu \qquad \text{and} \qquad E:(\hUCld)[a^{-1}]\lra \hUClts.
$$
Arguing exacly as in \cite[Lemmas 3.7 and 3.9]{FLMM} we see that $E$ is injective with dense image, while the specialisation map induces an isomorphism
$\hUCld /a (\hUCld) \simeq \hUClu$. Finally, we have an injective map $I:\hUClt\ra \hUClts$ induced by the natural inclusion $K_t\ra K_{t,s}=K_t\times K_s$ given by $f\mapsto (f,0)$. Similarly, we have an injective map $J:\hUCls\ra \hUClts$. As in Section 3.3 of \cite{FLMM}, the product of these maps $I\otimes J:(\hUClt) \otimes _Q (\hUCls)\ra\hUClts$ is injective with dense image.

\subsection{Fock module}We now describe the ``fermionic" Fock spaces corresponding to the Clifford algebras defined in the previous section. As above, for the construction in the case of one singularity we refer to \cite[Section 15.1.4]{FB}: here we mimic this definition in the case of two singularities. We define
$\Cl_2^+$ as the $A$-subalgebra of $\Cl_2$ generated by $R_2\otimes \gon_+$ and $R_2\otimes \gon_+^*$ and we define the Fock module 
$$ \Fock_{2}=\Cl_2\otimes_{\Cl_2^+} A\, \statovuotoFd $$
where $R_2\otimes \gon_+$ and $R_2\otimes \gon_+^*$ acts trivially on $\statovuotoFd$.
The charge (see equation \eqref{eq:charge}) induces a grading on the Fock space by setting 
$$ \cha\statovuotoFd=0. $$
We denote by $\Fock_2^n$ the subspace of homogeneous elements of charge equal to $n$. Similar constructions can be given for all the other Clifford algebras $\Cl_1,\Cl_t,\Cl_s$, and $\Cl_{t,s}$, giving Fock modules $\Fock_1$, $\Fock_t$, $\Fock_s$, and $\Fock_{t,s}$.

Specialisation and expansion, induce maps also at the level of the Fock spaces. Arguing as in \cite[Section 6]{FLMM} (where we considered the module $\mV^{\grl,\mu}_2$), it is easy to prove the following Lemma:

\begin{lemma}\label{lem:SpEFock}\hfill

\begin{enumerate}[\indent a)]
\item The specialisation map $\Sp:\Fock^\bullet_2\lra \Fock^\bullet_1$ is homogeneous of degree zero and induces an isomorphism $\Fock^\bullet_2/a\Fock^\bullet_2\simeq \Fock^\bullet_1$.
\item We have a homogeneous isomorphism of degree zero $\Fock^\bullet_{t,s}\simeq \Fock^\bullet_t\otimes_Q \Fock^\bullet_s$.
\item The expansion map $E:\Fock^\bullet_2[a^{-1}]\lra \Fock^\bullet_t\otimes_Q\Fock^\bullet_s$ is a homogeneous isomorphism of degree zero.
\end{enumerate}
\end{lemma}

Recall also that the Fock space $\Fock=\Fock_1$ has a natural structure of vertex superalgebra with the following universal property. 
\begin{lemma}\label{lemma: universal property Fock space}
Let $U$ be a vertex superalgebra such that there exists a linear map $x\mapsto u_x$ from $\gon^*_+\oplus \gon_+^*$ to the space of odd elements of $U$ such that 
\begin{enumerate}
\item for all $\grf,\psi\in \gon$ and for all $\grf^*,\psi^*\in \gon^*_+$ 
$$
(u_\psi)_{(n)}(u_\grf)=(u_{\psi^*})_{(n)}(u_{\grf^*})=(u_\psi)_{(m)}(u_{\psi^*})=(u_{\psi^*})_{(m)}(u_{\psi})=0
$$
for all $n\geq 0$ and for all $m\geq 1$;
\item $(u_\psi)_{(0)}(u_{\psi^*})=(u_{\psi^*})_{(0)}(u_{\psi})=\langle \psi,\psi^*\rangle \statovuoto_U$ for all $\psi\in \gon$ and $\psi^*\in \gon^*_+$.
\end{enumerate}
Then there exists a unique morhism of vertex superalgebras $\gra:\Fock\ra U$ such that $\gra(\psi t^{-1}\statovuotoF)=u_\psi$ and $\gra(\psi^* t^{-1}\statovuotoF)=u_{\psi^*}$.
\end{lemma}

\subsection{Bases}\label{sec:basi}
For each of the objects introduced above -- base rings, enveloping algebras, Clifford algebras, and Fock spaces -- it is not hard to construct explicit bases (or topological bases). We give the details in the case of two singularities. The construction of a basis depends on the choice of a basis of $\mC[t,s][1/ts]$ as an $A$-module. Following \cite{FLMM}, Section 1.1 and Equation (4.1) we introduce the following bases, indexed by $\tfrac12 \mZ$: for $n\in \mZ$ we define
$$
\begin{cases} z_n =t^ns^n \\ z_{n+\frac12 }=t^{n+1}s^{n} \end{cases}
\qquad 
\begin{cases} w_n =t^ns^n \\ w_{n+\frac12 }=t^{n}s^{n+1} \end{cases}
$$
The elements $z_m$ for $m\in \tfrac12 \mZ$ form a basis of $\mC[t,s][1/ts]$ as an $A$-module, and the elements $w_n$ are the dual basis with respect to the residue bilinear form: more precisely, one has
$$ \Res_2(z_n w_{-m-\tfrac 12})=\delta_{n,m}. $$
This specific choice of basis is not particularly important, and several others would be possible. However, some properties need to be satisfied for our approach to work.
In particolar with our choice, the elements $z_m$ (or $w_m$) with $m\geq 0$ form an $A$-basis of $\mC[t,s]$.

Since $K_2$ is an $A$-free module, we deduce that the enveloping algebras of $\gog_2$ and $\Cl_2$ are $A$-free modules. Moreover, as $R_2$ is a direct summand of $K_2$, we also deduce that $\mV^{\grl,\mu}_2$ and $\Fock_2$  are also $A$-free modules. Explicit bases of these modules, as well as an explicit topological basis of the algebra $\hU_2\hat \otimes _A\Cl_2$, can be obtained using the Poincar\'e-Birkhoff-Witt theorem and its analogue for Clifford algebras. 

\section{Vertex algebras and semi-infinite cohomology}\label{sect:Vertex algebras}
In this section, we recall some results obtained  by Casarin \cite{Casarin1}
which allow us to use the formalism of vertex algebras also in the context of several singularities. In particular, using this formalism  we develop a notion of semi-infinite cohomology for $\hU_2$-modules.

\subsection{Distributions and vertex algebra morphisms}\label{sec:distribuzioni}
Let $\calR$ be a complete topological associative $A$-algebra. Following \cite[Definition 3.0.4]{Casarin1},  we denote by $\field_A(K_2,\calR)$ the space of continuous $A$-linear morphisms from $K_2$ to $\calR$ and call it the \textit{space of 2-fields}. 
We refer to \cite{Casarin1} for the definitions of mutually local $2$-fields (Definition 3.1.1), of the $n$-products $X_{(n)}Y$ of two $2$-fields (Definitions 3.1.2 and 3.1.7) and of the derivative $\partial(X)$ of a $2$-field (before Definition 3.0.3). The definition in \cite{Casarin1} applies also to the other rings we are considering: $K_1,K_t,K_s,K_{t,s}$. 

In particular to define $n$ products it is necessary to choose what in \cite{Casarin1},  Definition 2.3.8, is called a global coordinate. We choose always $t$ as a global coordinate. More explicitly for the rings $K_2,K_1,K_t$ and $K_s$ we choose $t=s+a$ as a global coordinate, and for the ring $K_{t,s}=K_t\times K_s$ we choose $(t,t)=(t,s+a)$. 

We also use some foundational results proved in this context in \cite{Casarin1}. In particular, the following result will be crucial for us. 

\begin{theorem}[\cite{Casarin1}, Theorem 3.2.3]\label{teo:casarin} 
Let $\calF$ be a $\mC$-linear subspace of $\field_A(K_2,\calR)$ of mutually local $2$-fields closed under derivation and $n$-products. Let $\uno$ be a field such that $\uno(f)$ is central for every $f\in K_2$, that $\partial\uno =0$ and such that $\uno_{(n)}X=\delta_{n,-1}X$ for all $X\in \calF$. Then the vector space $\calF+\mC\uno$, endowed with $n$-products and derivation $T=\partial$, is a $\mC$-vertex algebra with $\uno$ as vacuum vector. 
\end{theorem}

It is straightforward to generalize the constructions and  results in \cite{Casarin1} to the case of superalgebras $\calR$. 

\medskip

We are interested in the case where $\calR$ is the superalgebra $\hU_2\hat\otimes_A \Cl_2$. For  $x\in \gog$, 
$\psi\in \gon_+$ and $\psi^*\in \gon_+^*$ we define the 2-fields
\begin{equation}\label{eq:campi}
x^{(2)}(g) = (x\otimes g)\otimes 1_{\Cl_2},
\quad  \psi^{[2]}(g)    = 1_{\hat U_2}\otimes (\psi\otimes g),
\quad (\psi^*)^{[2]}(g) = 1_{\hat U_2}\otimes (\psi^*\otimes g) 
\end{equation}
for all $g\in K_2$. The first of these fields has even parity with respect to the superalgebra structure, while the second and third ones are odd. These fields are mutually local. We consider the 
minimal $\mC$-linear subspace $\calF^{(2)}$ of $\hU_2\hat\otimes_A \Cl_2$ closed under $n$-products and derivation and containing the fields  \eqref{eq:campi}.  Moreover, we define
$$ \uno_2 (f)=\Res_2(f) \big(1_{\hU_2}\otimes 1_{\Cl_2}\big). $$
It is easy to check that this data satisfies the hypothesis of Theorem \ref{teo:casarin}. Therefore,
$\calV^{(2)}=\calF^{(2)}+\mC\uno_2$ has a structure of vertex superalgebra, and by the universal properties of the vertex algebra $\mV$ (Lemma \ref{lemma: universal property Weyl module superalgebra}) and of the vertex superalgebra $\Fock^\bullet$ (Lemma \ref{lemma: universal property Fock space}) it follows that there exists a morphism of vertex superalgebras
\begin{equation}\label{eq: Phi 2}
\Phi^{(2)}: \mV\otimes_\mC \Fock^\bullet \lra \calV^{(2)}.
\end{equation}
This homomorphism will allow us to easily introduce  many elements in $\calV^{(2)}$, hence also in $\hU_2\hat\otimes_A \Cl_2$. 

Similar constructions apply if the algebra $\hU_2\hat\otimes_A \Cl_2$ is replaced by the algebras $\hU_1\hat\otimes \Cl_1$, $\hU_t\hat\otimes_Q\Cl_t$, etc. Hence, we construct the fields $x^{(1)}$, $\psi^{[1]}$, $x^{(t)}$, $\psi^{[t]}$, the vertex superalgebras $\calV^{(1)}$, $\calV^{(t)}$, and homomorphisms of vertex algebras $\Phi^{(1)}:  \mV\otimes_\mC \Fock^\bullet \lra \calV^{(1)}$, $\Phi^{(t)}:  \mV\otimes_\mC \Fock^\bullet \lra \calV^{(t)}$, etc. 

Notice that we have a specialisation morphism 
$
\Sp_\field:\field_A(K_2,\hUCl 2)\lra \field_{\mC}(K_1,\hU_2\hat\otimes \Cl_1)
$
and an expansion map $E_\field:\field_A(K_2,\hU_2\hat\otimes_A \Cl_2)\lra \field_Q(K_{t,s},\hU_2\hat\otimes_Q \Cl_{t,s})$, determined by the conditions 
\[
\big(\Sp_\field(X)\big)(\Sp(f))=\Sp(X(f)) \quad \text{ and } \quad \big(E_\field(X)\big)(E(f))=E(X(f)).
\]
These maps commute with $n$-products and derivations and satisfy 
$\Sp_\field(\uno_2)=\uno_1$ and $E_\field(\uno_2)=\uno_{t,s}$.
Moreover, by construction they satisfy
\[
\Sp_\field(x^{(2)})=x^{(1)}\qquad \text{and} \qquad E_\field(x^{(2)})=x^{(t,s)}
\]
for $x\in \gog$. Similar relations hold for $\psi^{[2]}$ and $(\psi^*)^{[2]}$. 
This implies in particular that the homomorphisms $\Sp_\field$ and $E_\field$ restrict to homomorphisms of vertex algebras $\Sp:\calV^{(2)}\lra \calV^{(1)}$ and $E:\calV^{(2)}\lra \calV^{(t,s)}$ such that 
$$ \Sp\circ \Phi^{(2)}=\Phi^{(1)}\qquad E\circ \Phi^{(2)}=\Phi^{(t,s)}. $$
We can also describe the morphism $\Phi^{(2)}$ through the morphisms $\Phi^{(t)}$ and $\Phi^{(s)}$. Recall from the end of Section \ref{sec:Cliff} the maps $I,J$ from $\hUClt$ and $\hUCls$ to $\hUClts$. These maps induce maps at the level of fields 
$I_\field:\field_Q(K_t,\hUClt)\ra \field_Q(K_{t,s},\hUClts)$ and
$J_\field:\field_Q(K_s,\hUCls)\ra \field_Q(K_{t,s},\hUClts)$, given by
$$
I_\field(X)(f,g)=I(X(f)) \qquad \text{ and }\qquad  J_\field(X)(f,g)=J(X(g))
$$
for all $(f,g)\in K_t\times K_s=K_{t,s}$. The maps $I_\calF$ and $J_\calF$ preserve $n$-products,  commute with derivations, and satisfy $I_\calF(\uno_t)+J_\calF(\uno_s)=\uno_{t,s}$. Moreover we notice that $I(u)$ and $J(v)$ commute for all
$u\in \hUClt$ and $v\in \hU_s\hat\otimes_Q\Cl_s$. By the discussion in \cite[Section 7.2]{Casarin1}, this implies 
$$
I_\calF \circ \Phi^{(t)}+J_\calF \circ \Phi^{(s)}=\Phi^{(t,s)}.
$$
This is the only statement where it is relevant the choice of the global coordinate we have done in Section \ref{sec:distribuzioni}.

\subsection{Semi-infinite cohomology}\label{ssec:comologiasemi}
We now define a notion of semi-infinite cohomology for $\hU_2$-modules, in analogy with the analogous notion for $\hU_1$-modules described for example in \cite[Chapter 15]{FB}. To this end, we introduce some notation for elements in the vertex superalgebra $\mV\otimes \Fock^\bullet$.
As in the case of $\hU_1$, to describe these elements we choose a basis $J^a$ of $\gog$ compatible with the decomposition $\gog=\gon_-\oplus \got \oplus \gon_+$, where $\gon_+$ is the nilpotent radical of $\gob$ and $\gon_-$ is the radical of the opposite nilpotent borel subalgebra. We denote by $c^{b,d}_{e}$ the structure coefficients of the Lie bracket with respect to this basis. We denote by $\Phi\sqcup \Gamma$ the indexing set of the basis $J^a$, so that, if $\gra\in \Phi$, then $J^\gra=e_\gra=f_{-\gra}$ is a root vector of weight $\gra$ 
and, if $\gra \in \Gamma$, then $J^\gra \in \got$. We also denote by $\psi_\gra^*$ for $\gra\in \Phi^+$ the basis of $\gon_+^*$ dual to the basis $e_\gra$  of $\gon_+$.

With each element in $\gon_+\otimes \cdots \otimes \gon_+ \otimes \gon^*_+\otimes \cdots \otimes \gon^*_+$ we associate an element in the vertex superalgebra $\Fock$ as follows:
$$ N(\psi_1\otimes\dots\psi_\ell\otimes \psi^*_1\otimes \dots \otimes \psi_m^*)=(\psi_1 t^{-1})\cdots(\psi_\ell t^{-1})\cdot(\psi^*_1 t^{-1})\cdots(\psi_m^* t^{-1})\cdot\statovuotoF. $$

Similarly, with an element in $\gog\otimes\gon_+^*$ we associate an element in the vertex superalgebra $\mV\otimes \Fock^*$ by setting
$$ M(x\otimes \psi^*)=(x t^{-1})\cdot\statovuotoV\otimes (\psi^*t^{-1})\cdot\statovuotoF. $$
Following \cite[Chapter 15]{FB} we define
\begin{align*}
q= & M(I)-\frac 12 \statovuotoV\otimes N(B)=
\sum_{\gra\in\Phi^+}(e_\gra t^{-1})\cdot \statovuotoV \otimes (\psi^*_\gra t^{-1})\cdot\statovuotoF \\
&-\frac{1}{2}\sum_{\gra,\grb\in\Phi^+} c^{\gra,\grb}_{\gra+\grb}\;  \statovuotoV\otimes (e_{\gra+\grb} t^{-1})\cdot(\psi^*_\gra t^{-1})\cdot(\psi^*_\grb t^{-1})\cdot\statovuotoF,
\end{align*}
where $I\in \gog\otimes \gon_+^*$ represents the inclusion of $\gon_+$ in $\gog$ and $B\in \gon_+\otimes \gon_+^*\otimes \gon_+^*$ is the Lie bracket. We now define the boundary operator $d^{(2)}_{\std} \in \hU_2\hat\otimes_A \Cl_2$ as follows:
$$ d_{\std}^{(2)}:=\big(\Phi^{(2)}(q)\big)(1). $$

The boundary operator that we will use to define the semi-infinite cohomology is a deformation of $d_{\std}^{(2)}$. Let $\psi_{\pr}^*=\sum_{\gra \, \text{simple}}\psi^*_\gra\in \gon_+^*$, and define 
$$ \chi^{(2)}=1_{\hU_2}\otimes \psi_{\pr}^{*}=\Phi^{(2)}(N(\psi^*_{\pr}))(1) \in \hU_2\otimes_A \Cl_2. $$
Similar constructions yield  $\chi^{(s)}$, $\chi^{(t)}$, $\chi^{(s)}$, and $\chi^{(s,t)}$. Finally set
$$ d^{(2)}=d_{\std}^{(2)}+\chi^{(2)}.$$
As we will check in Section \ref{ssec:commutazioni}, this is an element that squares to zero, and therefore, it can be used to define the semi-infinite cohomology of a $\hU_2$-module. 

Similarly we can define $d_{\std}^{(1)}$, $\chi^{(1)}$, $d^{(1)}$, $d_{\std}^{(t)}$, $\chi^{(t)}$, $d^{(t)}$, and so on, as elements of the corresponding superalgebras. By the discussion at the end of Section 
\ref{sec:distribuzioni} we have 
$$
\Sp(d^{(2)})=d^{(1)}, \qquad 
E(d^{(2)})=d^{(t,s)}, \quad \text{ and } \quad 
I(d^{(t)})+J(d^{(s)})=d^{(t,s)}. 
$$


\begin{definition}\label{dfn:semiinfinitecohomology}
Let $M$ be an $\hU_2$ module. Consider the $\hU_2\hat\otimes_A\Cl_2$-graded module $M\otimes_A \Fock^\bullet_2$, where the grading is given by charge on $\Fock^\bullet_2$. The element $d^{(2)}$ acts on this module as a boundary operator of degree one. Define $\semiinf^n(M)$ as the corresponding cohomology of degree $n$.
\end{definition}

Similar constructions apply to modules over the algebras $\hU_1$, $\hU_t$, $\hU_s$ or $\hU_{t,s}$.

Let $Z_2$ be the center of the algebra $\hU_2$, and similarly introduce the center 
$Z_1$ of $\hU_1$ and the centers $Z_t$ and $Z_s$ of $\hU_t$ and $\hU_s$.
If $M$ is an $\hU_2$-module, the action of $Z_2$ on 
$M\otimes _A \Fock_2^\bullet$ commutes with the differential $d^{(2)}$ and preserves the charge, hence induces an action of $Z_2$ on the semi-infinite cohomology groups of $M$. A similar action is defined in the case of $\hU_1$-modules or $\hU_t$-modules.

Recall that a module $M$ over a topological algebra $\calR$ is said to be \textit{smooth} if the action of $\calR$ on $M$ is continuous with respect to the discrete topology on $M$.
Notice that, if $M$ is a smooth $\hU_2$-module, then, since the map $E$ has dense image, the action of $\hU_2$ on $M$ extends to a smooth action of $\hU_{t,s}$ on $M[a^{-1}]$. Similarly, if $M_t$ is a smooth $\hU_t$-module and $M_s$ is a smooth $\hU_s$-module, then there is an induced action of $\hU_{t,s}$ on $M_t\otimes_Q M_s$. In the next section we will use the following properties of the semi-infinite cohomology. 

\begin{lemma}\label{lem:SEcomologia}
\hfill
\begin{enumerate}[\indent a)]
\item Given a short exact sequence of $\hU_2$-modules, there is an induced long exact sequence in semi-infinite cohomology.

\item Let $M$ be an $\hU_1$-module. The semi-infinite cohomology of $M$ as an $\hU_1$-module is isomorphic to the semi-infinite cohomology of $M$ considered as an $\hU_2$-module through the map $\Sp$. 

\item Let $M$ be an $\hU_{t,s}$-module. The semi-infinite cohomology of $M$ as an $\hU_{t,s}$-module is isomorphic to the semi-infinite cohomology of $M$ considered as an $\hU_2$-module through the map $E$.  In particular, this applies to the case where $M=N[a^{-1}]$ is the localization of a smooth $\hU_2$-module $N$. 

\item Let $M_t$ be a smooth $\hU_{t}$-module, $M_s$ be a smooth $\hU_s$-module, and let $M := M_t\otimes_Q M_s$, regarded as a $\hU_{t,s}$-module. The complex computing the semi-infinite cohomology  of $M$ is the total complex associated with the double complex given by the tensor product of the complex  computing the semi-infinite cohomology of $M_t$ and that of $M_s$. 
In particular, being the base ring $Q$ a field,  if $M_t$ and $M_s$ have non zero semi-infinite cohomology only in degree zero, then $M$ considered as an $\hU_{t,s}$-module has semi-infinite cohomology only in degree zero and the cohomology in degree zero is isomorphic to the product of the tensor product of $\semiinfz(M_t)$ and $\semiinfz(M_s)$.
\end{enumerate}

\end{lemma}
\begin{proof}
Part a) follows from the fact that $\Fock_2$ is a free module over $A$. 	
	
Part b) follows from the fact that, since $a\in A$ acts trivially on $M$, by Lemma \ref{lem:SpEFock} a) we have 
$$M\otimes_A\Fock^\bullet _2\simeq M\otimes_\mC \frac{\Fock^\bullet _2}{a\Fock^\bullet_2}\simeq M\otimes_\mC \Fock^\bullet_1$$
and moreover, by construction, $d^{(1)} = \Sp(d^{(2)})$. 

Part c) follows from the fact that, since the action of $a$ on $M$ is invertible,  by Lemma \ref{lem:SpEFock} c) we have
$$
M \otimes_A \Fock^\bullet_2 = M \otimes_A \Fock^\bullet_2[a^{-1}]=M \otimes_A \Fock^\bullet_{t,s}
$$
and, moreover, by construction, $d^{(t,s)} = E( d^{(2)})$.

Finally, from Lemma \ref{lem:SpEFock} c) we have
$$
(M_t\otimes_Q \Fock_t^\bullet)\otimes_Q (M_s\otimes_Q \Fock_s^\bullet)\simeq M \otimes _Q \Fock_{t,s}^\bullet.
$$
Part d) then follows from the equality $d^{(t,s)}= I(d^{(t)})+J(d^{(s)})$.
\end{proof}

\subsection{Commutation relations}\label{ssec:commutazioni}
For their computation of the semi-infinite cohomology of $\mV$, Frenkel and Ben Zvi (see \cite{FB} Chapter 15) relied on the choice of a clever basis of $\mV\otimes\Fock$. For all $x\in \gog$, they define
$$
\hat x = xt^{-1}\cdot \statovuotoV \otimes \statovuotoF + N(\gra_x),
$$
where $\gra_x\in \gon_+\otimes\gon_+^*$ represents the linear map $\gon_+ \to \gon_+$ obtained as the composition of $\operatorname{ad}_x:\gon_+\lra\gon_+$, the natural projection 
$\pi : \gog\lra \gog/\gob_-$, and  the inverse of the isomorphism $\gon_+ \cong \gog/\gob_-$ induced by $\pi$.
Using the map $\Phi^{(2)}$ from Equation \eqref{eq: Phi 2} we define
$$
\hat x ^{(2)}=\Phi^{(2)}(\hat x).
$$
To compute the semi-infinite cohomology of $\mV_2^{\grl,\mu}$ we will need some information about the  commutation relations among the elements $\hat x^{(2)}$, $\psi^{[2]}$, and
$(\psi^*)^{[2]}$, and the boundary operators. These are easy to compute because all these objects are constructed through the map $\Phi^{(2)}$.
Let us make this remark precise. Given an element $x$ in $\mV\otimes \Fock$, denote 
by $x(z)$ the corresponding field in the vertex superalgebra and by $x^{(2)}:K_2\lra \hU_2\hat \otimes \Cl_2$ the 2-field $\Phi^{(2)}(x)$. 
For any choice of elements $x,y \in \mV\otimes \Fock$, the commutator of the corresponding fields is given by
$$
[x(z),y(w)]=\sum_{n\geq 0}\frac 1{n!}(x_{(n)}y)(w)\partial_w^n\delta(z-w). 
$$ 
We have a similar Operator Product Expansion formula for $2$-fields (see \cite{Casarin1}, Proposition 3.1.3)
$$
[x^{(2)}(f),y^{(2)}(g)]=\sum_{n\geq 0}\frac 1{n!}\left((x^{(2)})_{(n)}(y^{(2)})\right)(g\,\partial^n f ),
$$
where the product $(x^{(2)})_{(n)}(y^{(2))})$ is the product of 2-fields defined in \cite{Casarin1}. However, since $\Phi^{(2)}$ is a map of vertex algebras we get $(x^{(2)})_{(n)}(y^{(2)})=(x_{(n)}y)^{(2)}$. Hence, if we know the commutator of $x(z)$, $y(w)$, we immediately deduce that of $x^{(2)}$ and $y^{(2)}$.

Similar considerations apply when we want to compute $[x^{(2)}(1),y^{(2)}(g)]$ assuming we know the commutator of $x_{(0)}$ and 
$y(w)$. In this case, the usual OPE formula gives $[x_{(0)},y(w)]=(x_{(0)}y)(w)$, while the OPE formula for 2-fields gives 
$$
[x^{(2)}(1),y^{(2)}(g)]=\left((x^{(2)})_{(0)}(y^{(2)})\right)(g ).
$$
Using again the fact that $\Phi^{(2)}$ is a map of vertex algebras, we get
$$
[x^{(2)}(1),y^{(2)}]=\Phi^{(2)}\Bigg( \Big( [x_{(0)},y(w)](\statovuotoV \otimes\statovuotoF)\Big)|_{w=0} \Bigg).
$$
These formulas are enough to determine all commutation relations among the elements $\hat x^{(2)}$, $\psi^{[2]}$,
$(\psi^*)^{[2]}$ and the boundary operators from those obtained by Frenkel and Ben Zvi in \cite[Chapter 15]{FB}, without the need of any further computation. We summarise these results in Proposition \ref{proposition:commutazioni} below, which (in light of the above) follows from Sections 15.2.4 and 15.2.9 of \cite{FB}. In the statement, we denote by $e_{\pr}, h_{\pr}, f_{\pr}$ the $\gos\gol(2)$-triple   
such that $f_{\pr}=\sum_{\gra\text{ simple}}\grl_\gra f_\gra$, $\kappa(f_{pr},e_{\gra})=1$ for all simple root $\gra$ and $h_{pr}\in \got$.

\begin{proposition}\label{proposition:commutazioni}for all $x\in\gog$, $y\in \gob$, $z\in\gon_+$, $w\in\gob_-$, $\psi\in \gon_+$ and $\psi^*\in\gon^*_+$ we have:
\begin{alignat*}{3}
	&a) & (d^{(2)}_{\std})^2                       &=0, &[d^{(2)}_{\std},\chi^{(2)}]_+    &=0, \\
	&b) & (\chi^{(2)})^2                         &=0, &(d^{(2)})^2                    &=0,\\
	&c) &  [\chi^{(2)},\psi^{[2]}]_+     &=\langle\psi_{\pr}^* ,\psi\rangle\, \uno ,
	          &\qquad[\chi^{(2)},(\psi^*)^{[2]}]_+&=0,  \\
	&d) &[\chi^{(2)},\hat z^{(2)}]   &= 0, &[\chi^{(2)},\hat w^{(2)}]&=\sum_{\gra\, \in \Phi^+}\kappa([f_{\pr},z],e_\gra) \psi_\gra ^*,\\
	&e) &\; [d_{\std}^{(2)},\psi^{[2]}]_+         &=\hat\psi^{(2)},  &[d_{\std}^{(2)},(\psi^*)^{[2]}]_+ 
	&=-\frac 12 \Phi\big(1_{\hU_2}\otimes N(\psi^*\circ B)\big), \\
	&f) & [d_{\std}^{(2)},\hat y^{(2)}]_+         &=0
\end{alignat*}
where in the second formula of e) the element $\psi^*\circ B\in \gon_+^*\otimes \gon_+^*$ represents the composition of the bracket with the map $\psi^*$.
Moreover, if we choose a basis $J^a$ as at the beginnin of Section \ref{ssec:comologiasemi}, for all $\grg\in\Phi^+$ we have
\begin{align*}
[d_{\std}^{(2)},\hat f_\grg^{(2)}]_+= &\sum_{\gra \in \Phi^+, a\in \Phi^-\sqcup \grG} c^{\gra,-\grg}_{a} {(\hat J^{a})^{(2)}}_{(-1)}(\psi_\gra^*)^{[2]} \\ &-\frac 12 \, \kappa(e_{-\grg},f_\grg) \,\partial (\psi^*_{-\grg})^{[2]} -\sum_{\gra,\grb \in \Phi^+,\, a\in\Phi\sqcup \grG} c^{\gra,a}_{\grb} c^{\grb,-\grg}_{a} \partial (\psi^*_\gra)^{[2]}
\end{align*}
\end{proposition}
By specialisation and localization we obtain that similar formulas hold also in the case of our various other superalgebras $\hU_t\hat \otimes \Cl_t$, $\hU_{t,s} \otimes \Cl_{t,s}$, \ldots

\section{The semi-infinite cohomology of $\mV_2^{\grl,\mu}$}\label{sec:calcolocomologia}

In this section we compute the semi-infinite cohomology of $\mV^{\grl,\mu}_2$. We denote by $C^\bullet_2=C^\bullet_2(\grl,\mu)$ the complex $\mV^{\grl,\mu}_2\otimes _A \Fock^\bullet_2$ and similarly we introduce the complexes $C^\bullet_t=C^\bullet_t(\grl)=\mV^{\grl}_t\otimes _Q \Fock^\bullet_t$ and 
$C^\bullet_s=C^\bullet_s(\mu)=\mV^{\mu}_s\otimes _Q \Fock^\bullet_s$. We further introduce the complexes $C^\bullet_1(\nu)=\mV^{\nu}_1\otimes _\mC \Fock^\bullet_1$ and $C^\bullet_1(\grl,\mu)=\mW_1^{\grl,\mu}\otimes_\mC \Fock^\bullet_1$. Hence, we have $C^\bullet_1(\grl,\mu)\simeq\oplus C^\bullet_1(\nu)$, where the sum ranges over the irreducible factors of $V^\grl\otimes V^\mu$ counted with multiplicity. 

We denote by $\Op_1$ the indscheme of opers on the punctured disc and, for every integral dominant weight $\nu$, we write $\Op_1^\nu$ for the associated connected component of the space of unramified opers without monodromy, equipped with its reduced structure (see, for example, \cite{FG6} for a more complete definition). We also denote by $v_\nu$ a highest weight vector in the $\gog$-module $V^\nu$. 
Feigin and Frenkel \cite{FF92} constructed an isomorphism $\calF_1:\Funct(\Op_1)\lra Z_1$ between the space of functions over $\Op_1$ and the center $Z_1$ of $\hU_1$. Recall the following result, which combines Theorem 1, Theorem 2 and the proof of Proposition 1 in \cite{FG6}.

\begin{theorem}[Frenkel and Gaitsgory \cite{FG6}] \label{thm:FG6teo} 
The action of $Z_1$ on $\mV_1^\nu$ and the Feigin-Frenkel isomorphism induce an isomorphism
$$\calG_1: \Funct(\Op_1^\nu)\lra \End_{\hg_1}(\mV_1^\nu).$$
Moreover, the element $v_\nu\otimes \statovuotoF$ is a cocycle in $C^\bullet_1(\nu)$ and the map $z\mapsto [z\cdot v_\nu\otimes \statovuotoF]$ from $Z_1$ to $\semiinfz(\mV^\nu_1)$ induces isomorphisms of $Z_1$-modules
$$ \Funct(\Op_1^\nu)\simeq \End_{\hg_1}(\mV_1^\nu)\simeq \semiinfz(\mV_1^\nu). $$
Finally, $\semiinf^n(\mV_1^\nu)$ vanishes for all $n\neq 0$. 
\end{theorem}

The result of Frenkel and Gaitsgory generalises easily to the case of the modules $\mV^{\grl}_t$ and $\mV^{\mu}_s$.

By Lemma \ref{eq:SEWeyl} and Lemma \ref{lem:SpEFock}, as in the proof of Lemma \ref{lem:SEcomologia}, by the compatibility of boundary operators we get homomorphisms of complexes $\Sp:C^\bullet_2\ra C^\bullet_1(\grl,\mu)$ and $E:C^\bullet_2\ra
C^\bullet_t(\grl)\otimes_Q C^\bullet_s(\mu)$. These induce isomorphisms
\begin{equation}\label{eq:localizzazioneespecializzazionecomplesso}
C^\bullet_2[a^{-1}]\simeq  C^\bullet_t(\grl)\otimes_Q C^\bullet_s(\mu)\qquad\text{and}\qquad \frac {C_2^\bullet}{aC_2^\bullet}\simeq C^\bullet_1(\grl,\mu).
\end{equation}
From these isomorphisms and Theorem \ref{thm:FG6teo} it follows easily that $\semiinf^n(\Weyl)$ is zero for $n\neq 0,1$, and we could also get information on the cohomology in degrees zero and one.

However, it is easier to compute these cohomology groups directly by adapting the strategy employed by Frenkel and Ben Zvi in \cite[Chapter 15]{FB}. In order to do this, we now introduce certain subcomplexes of $C_2^\bullet$.  We denote by $1_{\mV^{0,0}_2}$ the element $1 \in A \otimes_\mC \mC \otimes_\mC \mC \subset \mV^{0,0}_2$. 

\begin{definition}
We denote by $E^\bullet_2$ the subcomplex of $C^\bullet_2(0,0)$ spanned by elements of the form 
\begin{equation}\label{eq:defD}
\hat x_1^{(2)}(g_1)\cdots \hat x_a^{(2)}(g_a)\cdot  1_{\mV^{0,0}_2}
\otimes
\psi_1^{(2)}(\ell_1)\cdots \psi_b^{(2)}(\ell_b) \cdot \statovuotoFd
\end{equation}
where $x_i,\psi_i\in\gon_+$ and $g_1,\dots,g_a, \ell_1, \dots\ell_b\in K_2$. By the commutation relations of Section \ref{ssec:commutazioni} we see that $E^\bullet_2$  is a subcomplex of $C^\bullet_2(0,0)$.
\end{definition}

We define also analogous complexes $E^\bullet_t$, $E^\bullet_s$ and $E^\bullet_1$. These complexes were denoted by $C'$ in \cite{FB} and by $C_0$ in \cite{FG6}. By construction, these subcomplexes are compatible with specialisation and localization, and 
there are isomorphisms $E^\bullet_2/aE_2^\bullet\simeq E^\bullet _1$ and $E^\bullet_2[a^{-1}]\simeq E^\bullet_t\otimes _Q E^\bullet_s$.

\begin{definition}
	We denote by $D^\bullet_2=D^\bullet_2(\grl,\mu)$ the subcomplex of $C^\bullet_2(\grl,\mu)$ spanned by elements of the form 
	\begin{equation}\label{eq:defE}
	\hat y_1^{(2)}(h_1)\cdots \hat y_c^{(2)}(h_c)\cdot  w
	\otimes
	(\psi^*_1)^{(2)}(k_1)\cdots (\psi_d^*)^{(2)}(k_d) \cdot \statovuotoFd
	\end{equation}
	where $w\in V^\grl\otimes V^\mu$, $y_i\in \gob_-=\gon_-+\got$, $\psi_i^*\in \gon_+^*$ and $h_1,\dots,h_c, k_1, \dots, k_d\in K_2$. By the commutation relations of Section \ref{ssec:commutazioni} we see that $E^\bullet_2$ is a subcomplex of  $C^\bullet_2(\grl,\mu)$. 

\end{definition}

We define also analogous complexes $D^\bullet_t(\grl)$, $D^\bullet_s(\mu)$ and $D^\bullet_1(\nu)$. These complexes were denoted by $C_0$ in \cite{FB} and by $C'$ in \cite{FG6}.
Finally, we denote by $D^\bullet_1(\grl,\mu)$ the analogous subcomplex of $C^\bullet_1(\grl,\mu)$. By construction, these subcomplexes are compatible with specialisation and localization, and there are isomorphisms $D^\bullet_2/aD_2^\bullet\simeq D^\bullet _1(\grl,\mu)$ and $D^\bullet_2[a^{-1}]\simeq D^\bullet_t(\grl)\otimes _Q D^\bullet_s(\mu)$.

\medskip

There is an isomorphism of complexes
$E_2^\bullet\otimes D_2^\bullet\lra C_2^\bullet $ defined by
$$
\big(\underline x \cdot 1_{\mV^{0,0}_2} \otimes \underline \psi \cdot \statovuotoFd \big)
\otimes 
\big(\underline y \cdot w \otimes \underline \psi^* \cdot \statovuotoFd \big)\longmapsto 
\underline x \cdot \underline y \cdot w\otimes 
\underline \psi\cdot \underline \psi^* \cdot \statovuotoFd,
$$
where  $\underline x=\hat x_1^{(2)}(g_1)\cdots \hat x_a^{(2)}(g_a)$ and
$\underline \psi=\psi_1^{(2)}(\ell_1)\cdots \psi_b^{(2)}(\ell_b)$ are as in Equation
\eqref{eq:defD}, $\underline y=\hat y_1^{(2)}(h_1)\cdots \hat y_{c}^{(2)}(h_c)$ and $\underline \psi^*=
(\psi^*)^{(2)}(k_1)\cdots (\psi^*)^{(2)}(k_d)$ are as in Equation \eqref{eq:defE}, and $w$ is an element of $V^\grl\otimes V^\mu$.

We now compute the cohomology of the complex $E^\bullet_2$. We will need the following result by Frenkel and Ben Zvi.

\begin{lemma}[{\cite[Section 15.2.6]{FB}}]\label{lem:coomD}
$H^n (E_1^\bullet)= 0$ for $n\neq 0$ and $\semiinfz(E_1^\bullet)=\mC [\statovuotoV\otimes \statovuotoF]
$.
\end{lemma}

This result generalizes easily to the case of $E_t^\bullet$ and $E_s^\bullet$. 
Localizing and specializing, we deduce the following lemma. 

\begin{lemma}\label{lem:coomE2}
	$H^n (E_2^\bullet)= 0$ for $n\neq 0$ and $H^0(E_2^\bullet)=A [1_{\mV^{0,0}_2} \otimes
	\statovuotoFd]
	$.
\end{lemma}
\begin{proof}
	By definition, the complex $E^\bullet_2$ is concentrated in non-positive degrees. Hence, the long exact sequence induced by
	$$
	\xymatrix{ 0 \ar[r] &  E^\bullet_2 \ar[r]^{a\cdot } & E^\bullet_2  \ar[r] & E^\bullet_1 \ar[r] & 0}
	$$
	implies that $H^n(E^\bullet_2)$ is torsion free for every $n$, and that the specialisation of $H^0(E^\bullet_2)$ is isomorphic to $H^0(E^\bullet_1)$. Since semi-infinite cohomology commutes with localization (Lemma \ref{lem:SEcomologia}), using Lemma \ref{lem:isoMN} and Lemma \ref{lem:coomD} we get the desired result. 
\end{proof}

We now compute the cohomology of $D^\bullet_2$. The strategy is similar, but the argument is less straightforward since we do not have an explicit representative for $H^0(D_1^\bullet)$. Following the strategy in \cite{FB}, we introduce the following bigraded structure on $D^\bullet_2$. Recall that the \textit{height} $\het(\gra)$ of a root $\gra$ is equal to the sum of the coefficients of $\gra$ when written as a sum of simple roots. Let also  $e_{\pr},h_{\pr},f_{\pr}$ be an $\gos\gol(2)$-triple 
such that $f_{\pr}=\sum_{\gra\text{ simple}}f_\gra$ and $h_{\pr}$ belongs to $\got$.

\begin{definition}\label{def:bigrado}
	We define a  bidegree, with values in $\tfrac12\mZ\times \tfrac12 \mZ$ and denoted by $\bigrado$, as follows.
    On elements of $\hgog_2$, we set
	\[	\bigrado (x\otimes g) =(-n,n)	\]
	if $x\in \gog$  is such that $[h_{\pr},x] = 2\,n\, x$ and $g\in K_2$. We set also the bidegree of the central element $C_2\in\hgog_2$ to be $(0,0)$. This induces a bidegree on $U(\hgog_2)$. On the space $X_2=K_2\otimes \gon_+\oplus K_2\otimes \gon_+^*$ (see Section \ref{sec:Cliff}) we define 
	$$ \bigrado e_\gra\otimes g = (-\het(\gra),-1+\het(\gra)) $$
    $$ \bigrado \psi_\gra^* \otimes g= (\het(\gra),1-\het(\gra)) 
	$$
	for $\gra$ a positive root and $g$ any element of $K_2$. This induces a bidegree on the Clifford algebra $\Cl_2$. Moreover, if $W$ is any finite-dimensional representation of $\gog$, then we set
$$
\bigrado w = (-n,n)
$$
if $w\in W$ is such that $h_{\pr}\cdot w= 2\, n\, w$. These choices induces a bidegree on the module
$C^\bullet_2(\grl,\mu)$, and the 
element $\hat x ^{(2)}(g)$ is homogeneous of bidegree $(-n,n)$ if  $[h_{\pr},x] = 2\,n\, x$.
Finally, notice that if an element has bidegree $(p,q)$, then it has charge  $p+q$. In particular, we introduce the submodule $D_2^{p,q}$ of elements of $D^{p+q}_2$ of bidegree $(p,q)$. 
\end{definition}

We notice also that $\bigrado d^{(2)}_{\std}=(0,1)$ and that $\bigrado \chi^{(2)}=(1,0)$. In particular, $D^{\bullet,\bullet}_2$ is a double complex and $D^\bullet_2$ is the associated total complex. Following Frenkel and Ben Zvi \cite[Chapter 15]{FB}, the cohomology of the rows of this double complex is easy to describe. Let $\goa$ be the centralizer of $f_{\pr}$ in $\gog$. Recall from \cite[Lemma 15.1.3 and Section 15.2.9]{FB} that the space spanned by monomials of the form 
$(\hat p_{1})_{n_1}\cdots (\hat p_{k})_{n_k}\cdot \statovuotoV\otimes \statovuotoF$ with $p_i\in \goa$ generates a commutative vertex subalgebra $F_1$ of $\mV\otimes \Fock^\bullet$ isomorphic to $S^\bullet (\goa\otimes t^{-1}\mC[t^{-1}])$. As in Section \ref{ssec:commutazioni}, it follows that for $x,y\in \goa$ the fields $\hat x^{(2)}$ and $\hat y^{(2)}$ commute. 

We define $F_2(\grl,\mu)$ as the span of elements of the form 
$$
\hat x^{(2)}_1(g_1)\cdots \hat x^{(2)}_k(g_k)\cdot (v\otimes \statovuotoFd)\in \Weyl\otimes_A\Fock^\bullet_2
$$
with $x_1,\dots,x_k\in\goa$ and $v\in V^\grl\otimes V^\mu$.
Notice that all these elements have charge equal to zero, and that the space  $F_2(\grl,\mu)$ splits 
as a direct sum $F_2(\grl,\mu)=\bigoplus_{q}F^{-q,q}_2(\grl,\mu)$ according to the bidegree introduced above. Moreover, by Proposition \ref{proposition:commutazioni} d), these elements are annihilated by the action of $\chi^{(2)}$. 

Similarly we construct subspaces $F_1^{-q,q}(\nu)\subset \mV^\nu_1\otimes_\mC \Fock_1^\bullet$, 
$F_t^{-q,q}(\grl)\subset \mV^\grl_t\otimes_Q \Fock_t^\bullet$,
$F_s^{-q,q}(\mu)\subset \mV^\grl_s\otimes_Q \Fock_s^\bullet$,
and 
$F_1^{-q,q}(\grl,\mu)\subset \mW^{\grl,\mu}_1\otimes_\mC \Fock_1^\bullet$, In particular, 
$F_1^{-q,q}(\grl,\mu)=\bigoplus_\nu F_1^{-q,q}(\nu)$ where the sum is over all irreducible factors of $V^\grl\otimes_\mC V^\mu$. By construction, the specialisation and localization maps induce isomorphisms
$$
\frac{F_2^{-q,q}(\grl,\mu)}{aF_2^{-q,q}(\grl,\mu)}\simeq F_1^{-q,q}(\grl,\mu)\quad \text{and} \quad
F_2^{-q,q}(\grl,\mu)[a^{-1}]\simeq \bigoplus_{b+c=q} F_t^{-b,b}(\grl)\otimes_Q F_s^{-c,c}(\mu).
$$

Recall the following result on the cohomology of $D_1^{\bullet,q}$ with respect to the boundary $\chi^{(1)}$.

\begin{lemma}[{\cite[Lemma 15.2.10]{FB} and \cite{FG6}}]\label{lem:righeD1} Let $2 p_\nu=\langle\nu,h_{\pr}\rangle $.
	\begin{enumerate}[\indent a)]
\item $D_1^{p,q}(\nu)=0$ for $q>p_\nu$ and for $p<-q$. In particular, $D_1^{p,q}=0$ for 
$q>p_{\grl+\mu}$ and for $p<-q$;
\item $H^n(D_1^{\bullet,q}(\nu))=0$ for $n\neq -q$. In particular, $H^n(D_1^{\bullet,q}(\grl,\mu))=0$ for $n\neq -q$;
\item The map $v\mapsto [v]$ from $F^{-q,q}_1(\nu)$ to $H^{-q}(D_1^{\bullet,q}(\nu))$ is an isomorphism. 
    \end{enumerate}
    Finally, it follows from c) that the map $v\mapsto [v]$ from $F_1^{-q,q}(\grl,\mu)$ to $H^{-q}(D_1^{\bullet,q}(\grl,\mu))$ is also an isomorphism.
\end{lemma}
Similar results hold for the complexes $D^{\bullet,q}_t(\grl)$ 
and $D^{\bullet,q}_s(\mu)$. From this result we deduce the cohomology of the complex $D^{\bullet,q}_q$ with respect to 
the boundary operator $\chi^{(2)}$.

\begin{lemma}\label{lem:righeD2} Let $2 p_0=\langle \grl+\mu,h_{\pr}\rangle$ as above.
	\begin{enumerate}[\indent a)]
	\item $D_2^{p,q}=0$ for $q>p_0$ and for $p<-q$; 
	\item $H^n(D_2^{\bullet,q})=0$ for $n\neq -q$;
	\item The map $v\mapsto [v]$ from $F_2^{-q,q}(\grl,\mu)$ to $H^{-q}(D_2^{\bullet,q}(\grl,\mu))$ is an isomorphism of $A$-modules. 
\end{enumerate}	
\end{lemma}
\begin{proof}Part a) is clear for the definition of $D_2^{p,q}=0$. For parts b) and c), we start by studying the localization of the cohomology groups of $D_2^{\bullet,q}$. Equivalently, we aim to compute the cohomology of the localization of the row $D_2^{\bullet,q}$. This localization can be rewritten as 
	\[ \bigoplus_{b+c=q} D_t^{\bullet,b}(\grl)\otimes D_s^{\bullet,c}(\mu). \]
	In particular, it follows from Lemma \ref{lem:righeD1} that its cohomology is concentrated in degree $-q$, and that its cohomology in this degree is given by
	\[ \bigoplus_{b+c=q} F^{-b,b}_t(\grl)\otimes F_s^{-c,c}(\mu), \]
   which is the localization of  $ F^{-q,q}_2(\grl,\mu)$.
	Since specialisation is compatible with $\bigrado$, we have an isomorphism $D^{\bullet,q} _2/aD^{\bullet,q} _2 \simeq D^{\bullet,q} _1(\grl, \mu)$. Using Lemma \ref{lem:righeD1}, the associated long exact sequence shows that $H^n(D_2^{\bullet,q})$ is torsion-free for $n\neq -q+1$, and that the map $$\iota: H^{-q}(D_2^{\bullet,q})/a H^{-q}(D_2^{\bullet,q}) \to H^{-q}(D^{\bullet,q} _1(\grl, \mu))$$
	is injective.
	
	We now prove c). Notice that both $F_2^{-q,q}(\grl,\mu)$ and $H^{-q}(D_2^{\bullet,q}(\grl,\mu))$ are torsion-free. We have already shown that the localization of the natural maps between them is an isomorphism. To study its specialisation, we compose it with the injection $\iota$. This composition is the isomorphism of the last remark of Lemma \ref{lem:righeD1}. We conclude by applying Lemma \ref{lem:isoMN}.
	
	In order to prove b), it is enough to notice that from the above discussion we know that, for $n\neq -q$, the module $H^n(D_2^{\bullet,q})=0$ is torsion-free, and that its localization is trivial.
\end{proof}	
	
Let now be $\grf_i^{(q)}$ be an $A$-basis of $F_2^{-q,q}(\grl, \mu)$. 
Since the cohomology in degree $-q$ of the complex $D_2^{\bullet,q+1}$ is zero, there exists an element $\grf_{i,1}^{(q)} \in D_2^{-q-1,q+1}$ such that $\chi^{(2)}(\grf_{i,1}^{(q)})=
-d_{\std}^{(2)}(\grf_i^{(q)})$. By induction, we can construct elements $\grf_{i,0}^{(q)}= \grf_{i}^{(q)}$ and $\grf_{i,\ell}^{(q)}\in D_2^{-q-\ell,q+\ell}$ such that their sum
$$
\tilde \grf_i^{(q)}=\sum _{\ell=0} ^{p_0-q} \grf_{i,\ell}^{(q)}
$$
satisfies $d^{(2)}(\tilde \grf^{(q)}_i)=0$. We now prove the main result of this section.

\begin{theorem}\label{teo:comologiaWeyl2}The following hold.
\begin{enumerate}[\indent a)]
	\item $\semiinf^n(\mV_2^{\grl,\mu})=0$ for $n\neq 0$. 
	\item We have an isomorphism 
	$$
	\frac {\semiinfz(\mV_2^{\grl,\mu})} {a\semiinfz(\mV_2^{\grl,\mu})}\simeq 
	\semiinfz(\mW_1^{\grl, \mu})\simeq \bigoplus _\nu \semiinfz(\mV_1^\nu)
	$$
	where the sum ranges over all irreducible components $V^\nu$ of $V^\grl\otimes V^\mu$, counted with multiplicity. 
	\item The elements $\big[\tilde \grf_i^{(q)}\big]$ are an $A$-basis of $\semiinfz(\mV^{\grl,\mu})$.
\end{enumerate}
\end{theorem}

\begin{proof}
	From Lemma \ref{lem:righeD2} we deduce that the classes of the elements $\tilde \grf_i^{(q)}$ form an $A$-basis of $H^0(D_2^\bullet)$, and that $H^n(D_2^\bullet)=0$ for $n\neq 0$. As the complex $D_2^\bullet$ is concentrated in non-negative degrees, by a standard homological argument we deduce that $H^n(\mV_2^{\grl,\mu})$ is isomorphic to the $n$-th cohomology of the complex $H^0(D_2^\bullet) \otimes_A E_2^\bullet$. Using Lemma \ref{lem:coomE2}, we immediately obtain parts a) and c).
	
	The second isomorphism appearing in part $b)$ is clear, while the first follows from a) and the long exact sequence associated with the isomorphism
	\[\frac {C_2^\bullet}{aC_2^\bullet}\simeq C^\bullet_1(\grl,\mu).\qedhere\]
\end{proof} 

We will use the following Corollary in the next Section. 
\begin{corollary}\label{cor:indivisibilitavv}
The element $[v_\grl\otimes v_\mu]\in \semiinfz(\mV^{\grl,\mu})$
is indivisible.
\end{corollary}
\begin{proof}
By the previous theorem we can choose $[v_\grl\otimes v_\mu]$ as an element of a basis of the free $A$ module $\semiinfz(\mV^{\grl,\mu})$.
\end{proof}

\section{The action of the center}\label{sec:azionecentro}
In this section we study the action of the center $Z_2$ on the semi-infinite cohomology of the module $\Weyl$.

In this section we show that $\mV_2^{\grl,\mu}$ is not a perfect analogue of the Weyl module $\mV_1^{\nu}$. Indeed, we show that, as a $Z_2$-module, the semi-infinite cohomology of $\mV_2^{\grl,\mu}$ is not isomorphic to $\End_{\hg_2}(\mV_2^{\grl,\mu})$
or to $\Funct(\Op_2^{\grl,\mu})$.

We begin by observing that the module $\semiinfz(\mV_1^{\nu})$ has no non-trivial $Z_1$-equivariant automorphisms.

First we notice, that by construction, the action of $Z_2$ commutes with localization and specialisation, as introduced before Equation \eqref{eq:localizzazioneespecializzazionecomplesso}. Concretely, we have:
$$
E_t(z\cdot x)=E_t(z)\cdot E_t(x),\quad
E_s(z\cdot x)=E_s(z)\cdot E_s(x), \quad 
\Sp(z\cdot x)=\Sp(z)\cdot \Sp(x)
$$
for all $z\in Z_2$ and for all $x\in \semiinfz(\mV_2^{\grl,\mu})$. 

\begin{lemma}\label{lem:unicitaisosucampo}
	
If $\calK:\End_{\hg_t}(\mV^{\grl}_t)\otimes _Q\End_{\hg_s}(\mV^{\mu}_s) \lra \semiinfz(\mV_t^\grl)\otimes_Q \semiinfz(\mV_s^\mu)$ is a $(Z_t\otimes Z_s)$-equivariant isomorphism, then $\calK(\Id_{\mV^\grl_t }\otimes\Id_{\mV^\mu_s })=q[v_\grl]\otimes [v_\mu]$ for some $q\in Q\senza\{0\}$.

\end{lemma}
\begin{proof}
It follows from Theorem \ref{thm:FG6teo} that $\End_{\hg_{t}\times \hg_s}(\mV^{\grl_t}\otimes_Q \mV^{\mu_s})$ is isomorphic to 
$\Funct(\Op_t^{\grl}\times_{\Spec Q}\Op_s^{\mu} )=\Funct(\Op_t^{\grl})\otimes_Q\Funct(\Op_s^{\mu})$ and this is a polynomial ring in infinitely many variables over the field $Q$. In particular, its only invertible elements are the non-zero scalars in $Q$. 

Moreover, Theorem \ref{thm:FG6teo} also implies that $\Funct(\Op^{\grl}_t)$ is isomorphic as a $Z_t$-module to $\semiinfz(\mV_t^\grl)$, with an isomorphism given by 
$ z \lra \calG_t (z) \cdot[v_\grl]$.
The claim follows.
\end{proof}

Before proving that $\Weyl$ does not have the ``right'' semi-infinite cohomology we recall some properties of the modules $\mV^\nu_1$ that will be needed also in the next section.
\begin{remark}\label{rmk:Opnudisgiunti}
We denote by $Z_1^\nu$ the coordinate ring of the scheme $\Op^\nu_1$. Recall that the schemes $\Op_1^{\nu}$ for different values of $\nu$ are disjoint, so that the map $Z_1\lra Z_1^{\nu_1}\times \dots \times Z_1^{\nu_k}$ is surjective if the weights $\nu_i$ are distinct. Recall also that the ring $Z^\nu_1$ is a polynomial ring in infinitely many variables. This implies that 
\begin{enumerate}
	\item There are no nontrivial $\hgog_1$-morphisms between the $\hU_1$-modules $\mV^\nu_1$ and $\mV^{\nu'}_1$ if $\nu\neq \nu'$.
	\item There are no nontrivial extensions between the $\hU_1$-modules $\mV^\nu_1$ and $\mV^{\nu'}_1$ if $\nu\neq \nu'$.
	\item Assume that $\gra : \prod Z^{\nu_i} \lra \prod Z^{\nu_i}$ is a map of $Z$-modules  and that the weights $\nu_i$ are distinct.
	If $1$ is in the image of $\gra$ then $\gra$ is an isomorphism and $\gra(Z_1^{\nu_i})= Z_1^{\nu_i}$. 
\end{enumerate}
\end{remark}

By the Feigin-Frenkel Theorem (see \cite{FLMM} Theorem 5.2) the ring $\Funct(\Op_2)$ is isomorphic to $Z_2$. In the sequel we will identify these rings through this isomorphism. In particular the ring $\Funct(\Op_2^{\grl,\mu})$ is a quotient 
of $Z_2$. We will denote $\Funct(\Op_2^{\grl,\mu})$ by $Z_2^{\grl,\mu}$. 

We now prove that $Z_2^{\grl,\mu}$ and $\semiinfz(\mV_2^{\grl,\mu})$ are not isomorphic.

\begin{proposition}\label{prop:noniso}
	Assume that $V^\grl\otimes V^{\mu}$ is not irreducible. 
	Then the two $Z_2$-modules $\End_{\hg_2}(\mV_2^{\grl,\mu})$ and $\semiinfz(\mV_2^{\grl,\mu})$ are not isomorphic. Similarly the two $Z_2$-modules $Z_2^{\grl,\mu}$ and $\semiinfz(\mV_2^{\grl,\mu})$ are not isomorphic.
\end{proposition}

\begin{proof}
Suppose $\calH :\End_{\hg_2}(\mV_2^{\grl,\mu})\lra \semiinfz(\mV_2^{\grl,\mu})$ is a $Z_2$-equivariant isomorphism. 
	
Recall from Lemma 4.28 in \cite{FLMM} that $Z_2[1/a]$ is dense in $Z_{t,s}$, and therefore the localization of $\calH$ is a $(Z_t \otimes_Q Z_s)$-equivariant isomorphism
\[ \End_{\hg_t}(\mV^{\grl}_t)\otimes _Q\End_{\hg_s}(\mV^{\mu}_s) \lra \semiinfz(\mV_t^\grl)\otimes_Q \semiinfz(\mV_s^\mu),\]
where we used the identification of the localization of $\semiinfz(\Weyl)$ with $\semiinfz(\mV_t^\grl)\otimes_Q \semiinfz(\mV_s^\mu)$.

From Lemma \ref{lem:unicitaisosucampo} and \ref{cor:indivisibilitavv} we deduce that $\calH (\Id_{\mV^{\grl,\mu}_2})=[q \,v_\grl\otimes v_\mu]$, where $q \in A$ and $qv_\grl\otimes v_\mu\in \Weyl$. We set $w=qv_\grl\otimes v_\mu\in \Weyl$.

By specialisation, $\calH$ gives a $Z_1$-equivariant isomorphism
\begin{equation}\label{eq: H bar}
\overline{ \calH } : \frac{\End_{\hg_2}(\mV_2^{\grl,\mu})}{a\End_{\hg_2}(\mV_2^{\grl,\mu})}\lra \frac{\semiinfz(\mV_2^{\grl,\mu})}{a\semiinfz(\mV_2^{\grl,\mu})}.
\end{equation}
This isomorphism sends $\overline{\Id_{\mV^{\grl,\mu}_2}}$ to $\overline{w}$. Now consider the decomposition $V^\grl\otimes V^\mu=\bigoplus V^\nu$ as $\gog$-modules. 
By Theorem \ref{teo:comologiaWeyl2}, the target of the map $\overline{\calH}$ in \eqref{eq: H bar} decomposes as $\bigoplus \semiinfz(\mV^\nu_1)$. The element $w$ is a multiple of $v_\grl\otimes v_\mu$ hence its class belongs to $\Psi^0(\mV_1^{\grl+\mu})$.
As $\overline{\calH}$ is $Z_1$-equivariant and $\mV^{\grl+\mu}_1$ is stable by the action of $\hgog_1$, we get that the image of $\overline{\calH}$ is contained in the direct summand $\semiinfz(\mV_1^{\grl+\mu})$. In particular, if $V^\grl\otimes V^\mu$ is not irreducible, the map $\overline{\calH}$ cannot be surjective. This proves the first claim. The second claim follows since the map from 
$Z_2^{\grl,\mu}$ to $\semiinfz(\Weyl)$ factors through $\End_{\hg_2}(\mV_2^{\grl,\mu})$. \end{proof}

\section{A Weyl module for $\gos\gol(2)$}\label{sec:tWeyl}
In this Section, we propose an alternative Weyl module in the context of opers with two singularities, in the case of $\gog=\gos\gol(2)$.  We fix the following notation: $e, h, f$ is an $\gos\gol(2)$-triple such that $h\in \got$ and $e\in \gon_{+}$, while $\psi^*\in\gon_+^*$ is the dual of $e$. We identify dominant weights with natural numbers and we assume from now on that $\grl\geq \mu$. In this case, the differential of the complex computing semi-infinite cohomology takes the simpler form $d^{(2)}=\psi^*+\sum  e w_n\otimes \psi^*z_{-n-1/2}$.

\medskip

Let $\tWeyl$ be the $\hU_2$-submodule of $\mV^{\grl,\mu}_2$ generated by the highest weight vector $1\otimes v_\grl\otimes v_\mu\in A \otimes V^\grl\otimes V^\mu$. We will prove that this module has the ``correct'' semi-infinite cohomology and the ``correct'' endomorphism ring. 

We start by giving a more explicit description of the module $\tWeyl$. If $X$ is a subspace of $U(\gog)$ and $Y$ is a subspace of a $\gog$-module $Z$ we denote by $X\cdot Y$ the subspace of $Z$ generated by the products $x\cdot y$ with $x\in X$ and $y\in Y$. We define an increasing filtration $F^i$ of $\tWeyl$ as follows
$$ F^i=U(\gog) \cdot (\mC\,\Id\otimes \Id + \Id \otimes \gog)^{i} \cdot (v_\grl\otimes v_\mu).$$ 
This is an increasing filtration of $V^\grl\otimes V^\mu$ by $\gog$-modules and for $i$ large enough we have $F^i=V^\grl\otimes V^\mu$. Choose a $\gog$-stable complement $G^{i+1}$ of $F^i $ in $F^{i+1}$ and set $G^0=F^0$, so that $F^i=\bigoplus_{j=0}^i G^j$. If we set $F^i(V^\mu)=(\mC\Id + \gon_-)^i v_\mu$, it is easy to check by induction on $i$ that 
$$ F^i=U(\gog)\cdot (\Id\otimes \Id + \Id \otimes \gon_-)^i (v_\grl\otimes v_\mu)=U(\gog)\cdot \big( \mC v_\grl\otimes F^i(V^\mu)\big). $$
In the case of $\gog=\gos\gol(2)$ we have $G^i\simeq V^{\grl+\mu-2i}$ and $F^\mu=V^\grl\otimes V^\mu $.

Let $U_2^-\subset U(\hgog_2)$ be the $A$-span of Poincar\'e-Birkhoff-Witt monomials of the form $(x_1w_{a_1})\cdots (x_kw_{a_k})$ with $x_i\in \gog$ and $a_i<0$. This is a complement of $U(\hgog_2^+)$ in $U(\hgog_2)$, so that in particular we have
$$
\Weyl=U_2^-\otimes _\mC (V^\grl\otimes V^\mu).
$$

\begin{lemma} \label{lem:descrizionetWeyl}
If $\grl\geq \mu$ then 
$$
\tWeyl = \sum _{i=0}^\mu
a^i U^-_2 \otimes_\mC  F^i =\bigoplus_{i=0}^\mu a^i U^-_2 \otimes_\mC  G^i
$$
\end{lemma}
\begin{proof}
To understand the module $\tWeyl$ we need to compute the 
$\hgog_2^+$-submodule of $A\otimes_\mC V^\grl \otimes_\mC V^\mu$
generated by $1\otimes v_\grl \otimes v_\mu$. Notice that 
every element of the form $xg$, with $x\in \gog$ and $g \in \mC[[t,s]]$ divisible by $ts$, acts trivially on $A\otimes V^\grl\otimes V^\mu$. Hence we need to understand the action of elements of the form
$$
z=x_1\cdots x_\ell \cdot (y_1t)\cdots (y_mt) \cdot (v_\grl\otimes v_\mu),
$$
with $x_i,y_i\in \gog$. Moreover, elements of $\gog$ act in the standard way on the tensor product $V^\grl\otimes V^\mu$, while elements of the form $xt$ with $x\in \gog$ act via $-a(\Id\otimes x)$. This implies the lemma.
\end{proof}

We now describe the specialisation of the module $\tWeyl$. We introduce the following decreasing filtration of $\tWeyl$:
\begin{equation} \label{eq:filtrazionetWeyl}
\mF_i=\tWeyl \cap a^i \Weyl.
\end{equation}
By Lemma \ref{lem:descrizionetWeyl} we have the following description of the terms of this filtration as $A$-modules:
$$
\mF_i= a^iU_2^- \otimes_\mC F^i \oplus \bigoplus _{j=i+1}^\mu a^jU_2^- \otimes_\mC G^j
$$
In particular we have $\mF_0=\tWeyl$, $\mF_j=a^j\Weyl$ for $j\geq \mu$. 

\begin{lemma}\label{lem:SptWeyl} \begin{enumerate} [\indent a)]
\item 	Let $u_i\in G^i$ be the highest weight vector and set $\tilde w_i =a^iu_i$. Then $\tilde w_i\in \tWeyl$
 and $a^{i-1}u_i\notin  \tWeyl$. 
\item There is an isomorphism of $\hU_1$-modules
$$
\frac{\mF_i+a\tWeyl}{a\tWeyl}\simeq \bigoplus_{j=i}^\mu\mV_1^{\grl+\mu-2j}.
$$
The quotient $\frac{\mF_i+a\tWeyl}{a\tWeyl}$ is generated as a $\hU_1$-module by the classes of $\tilde w_i,\dots,\tilde w_\mu$. 
In particular $\tWeyl/a\tWeyl\simeq \mW_1^{\grl,\mu}$ is generated by $\tilde w_0,\dots,\tilde w_\mu$.
\end{enumerate}
\end{lemma}
\begin{proof}The first claim follows from Lemma \ref{lem:descrizionetWeyl}.
		
	We prove part b) by decreasing induction on $i$. By Lemma \ref{lem:descrizionetWeyl}, for $i>\mu$ the quotient is zero and the claim is true. For $i\leq \mu$, consider the map
	$$
	U_2^- \otimes G^i \lra \frac{\mF_i + a^{i+1}\Weyl}{a^{i+1}\Weyl + \mF_i \cap a\tWeyl} \simeq  \frac{(\mF_{i}+a\tWeyl )/ a\tWeyl}{(\mF_{i+1}+a\tWeyl)/ a\tWeyl}
	$$
	sending an element $u\otimes v$ to the class of $a^iu\otimes v$. This map induces an isomorphism
	\begin{equation}\label{eq:isomFiltr}
	 \frac{U_2^-}{aU_2^-}\otimes G^i \simeq \frac{(\mF_{i}+a\tWeyl )/ a\tWeyl}{(\mF_{i+1}+a\tWeyl)/ a\tWeyl}.
\end{equation}
Moreover, notice that
	$\frac{U_2^-}{aU_2^-}\otimes G^i \simeq U_1^- \otimes G^i$, where $U_1^-=U(t^{-1}\gog[t^{-1}])\subset U(\hgog_1)=U_1$, and that $U_1^- \otimes G^i$ has a natural structure of $U_1$-module, as it can be identified with $\mV^{\grl+\mu-2i}_1$. With this $U_1$-action, the isomorphism $\ref{eq:isomFiltr}$ is $U_1$-equivariant.
	 Now the claim follows by the inductive hypothesis, combined with the fact that there are no nontrivial extensions between modules $\mV^\nu_1$ and $\mV^{\nu'}_1$ if $\nu\neq \nu'$ and that the highest weight vector of $V^\nu$ generates the module $\mV^\nu_1$ as an $U_1$-module. 
\end{proof}

Notice that, although the specialisations at $a=0$ of $\Weyl$ and $\tWeyl$ are isomorphic, the specialisation of $\tWeyl$,  is generated by $v_\grl\otimes v_\mu$ while in the first case this vector generates the submodule $\mV^{\grl+\mu}_1$.

As a corollary, we get the following result.

\begin{proposition}\label{prp:comologiatWeyl1}The following hold:

\begin{enumerate}[\indent a)]
\item $\semiinf^n(\tWeyl)=0$ for $n\neq 0$. 
\item The inclusion of $\tWeyl$ in $\Weyl$ induces isomorphisms $$\semiinfz(\tWeyl)[a^{-1}]\simeq \semiinfz(\Weyl)[a^{-1}] \simeq \semiinfz(\mV^\grl_t)\otimes _Q \semiinfz(\mV^\mu_s).$$
\item $\semiinfz(\tWeyl)$ is torsion-free with respect to the action of $A$, and the natural projection induces isomorphisms
$$
\frac{\semiinfz(\tWeyl)}{a\semiinfz(\tWeyl)} \simeq \semiinfz\left(\frac{\tWeyl}{a\tWeyl}\right)\simeq \semiinfz(\mW_1^{\grl,\mu}).
$$
\end{enumerate}
\end{proposition}
\begin{proof}
We use the filtration introduced in Equation \eqref{eq:filtrazionetWeyl}. Notice that 
$$
\frac{\mF_i}{\mF_{i+1}}=\frac {a^i U_2^-\otimes F^i}{{a^{i+1} U_2^-\otimes F^i}}\simeq U^-_1\otimes_\mC F^i\simeq \Ind^{\hgog_1}_{\hgog_1^+}F^i,
$$	
where we consider $F^i$ as a $\hgog_1^+$-module on which $t\gog[t]$ acts trivially. Notice that $\Ind^{\hgog_1}_{\hgog_1^+}F^i$ is a sum of modules of the form $\mV^\nu_1$, hence in particular has trivial non-zero cohomology. 

Hence, arguing by decreasing induction on $i$, starting from $i=\mu$, it follows that 
$\mF_i$ has trivial semi-infinite cohomology in degree different from zero. Indeed for $i=\mu$ we have $\mF_\mu=a^\mu\Weyl\simeq \Weyl$
and this is the content of Theorem \ref{teo:comologiaWeyl2}.
For $i=0$ this implies claim a).

Part b) follows from the fact that semi-infinite cohomology commutes with localization (see Lemma \ref{lem:SEcomologia})
combined with the isomorphism $\tWeyl[a^{-1}]=\Weyl[a^{-1}]\simeq \mV^\grl_t\otimes_Q\mV^\mu_s$. 
 
To prove c), consider the exact sequence
$$
\xymatrix{
 0\ar[r]
&\tWeyl \ar[r]^{\cdot a}
&\tWeyl \ar[r]
&\frac{\tWeyl}{a\tWeyl}\ar[r]
&0}
$$
By Lemma \ref{lem:SptWeyl}, the last module in this sequence is isomorphic to $\mW^{\grl,\mu}_1$. In particular, the semi-infinite cohomology groups $\semiinf^n$ of the modules appearing in this sequence are zero for $n \neq 0$, and c) follows. 
\end{proof}

To prove that the semi-infinite cohomology of $\tWeyl$ is isomorphic to $Z_2^{\grl,\mu}$
we will use the action of a particular central element in $Z_2$. Recall from \cite{FLMM} the definition of the 2-Sugawara operator
\begin{equation}\label{eq:Sug12}
S^{(2)}_{1/2}= \sum_{n\in\tfrac 12\mZ,b} :\ (J^b w_n)(J_b z_{-n})\ :
\end{equation}
where $J^1, J^2, J^3$ are the basis elements $e, h, f$ and $J_1, J_2, J_3$ are the dual basis elements $f, h/2, e$. As proved in \cite{FLMM}, the element $S^{(2)}_{1/2}$ is central. Its specialisation is the Sugawara operator
\begin{equation}\label{eq:Sug11}
S_1^{(1)} = \sum_{n \in \mZ, b} :\ (J^bt^n)\, (J_bt^{-n})\ :
\end{equation} 
which is an element of $Z_1$. It is straightforward to check that the action of $S_1^{(1)}$ on the Weyl module  $\mV^\nu_1$ is given by multiplication by $\nu(\nu+1)$.

\begin{lemma}\label{lem:calcoloSugawara}
	The element $\hat w_\ell =  \left(et^{-1}\right)^\ell \tilde w_\ell$ belongs to  $Z_2 \cdot (v_\grl\otimes v_\mu) + a\tWeyl$ for $\ell=0,\dots,\mu$,
\end{lemma}

\begin{proof}
	We notice first that the element $v_\grl\otimes f^\ell v_\mu$ belongs to $F^\ell\setminus F^{\ell-1}$ and has weight $\grl+\mu-2\ell$. Hence, up to a non-zero constant we have
	$v_\grl\otimes f^\ell v_\mu =u_\ell + u'_\ell$, where we recall that $u_\ell$ is the highest weight vector in $G^\ell\simeq V^{\grl+\mu-2\ell}\subset V^\grl\otimes V^\mu$ and $u'_\ell\in F^{\ell-1}$. In particular, recall from Lemma \ref{lem:SptWeyl} that $a^{\ell-1}F^\ell\subset \tWeyl$, hence
	$$
	a^\ell \left(et^{-1}\right)^\ell v_\grl\otimes f^\ell v_\mu = 
	\left(et^{-1}\right)^\ell \tilde w_\ell +\left(et^{-1}\right)^\ell (a^\ell u'_\ell)\equiv  	\left(et^{-1}\right)^\ell \tilde w_\ell \; \mod a\tWeyl.
	$$
	Hence, the lemma is equivalent to the fact that $\hat w_\ell=	a^\ell \left(et^{-1}\right)^\ell v_\grl\otimes f^\ell v_\mu$ is in $Z_2\cdot v_\grl\otimes v_\mu+a\tWeyl$. We prove this statement by induction on $\ell$. For $\ell=0$ it is trivially true. Now assume $\hat w_\ell$ is in $Z_2 \cdot v_\grl\otimes v_\mu + a\tWeyl$. We compute $ S^{(2)}_{1/2} (\hat w_\ell)$. In order to do this, we notice that the action of $xt^is^j$ on $\tWeyl/a\tWeyl$ is equal to the action of $xt^{i+j}$ on the same module, and that $ v_\grl\otimes e\,f^\ell v_\mu$ is in $F^{\ell-1}$. We have
    \begin{align*}
    S^{(2)}_{1/2} \hat w_\ell =\;  
      &2\sum _{n>0}  et^{-n}\cdot ft^{n} \cdot \hat w_\ell
    +2\sum _{n>0} ft^{-n}\cdot et^{n} \cdot \hat w_\ell
    + \sum _{n>0} ht^{-n}\cdot ht^{n} \cdot \hat w_\ell\\
    & +  e\cdot f\cdot \hat w_\ell+e\cdot f\cdot \hat w_\ell+\frac12h\cdot h \cdot \hat w_\ell.
    \end{align*}
    In the second infinite sum above, the element $et^n$ commutes with $et^{-1}$, hence $et^{n} \cdot \hat w_\ell\in a\tWeyl$ for all $n>0$. The summands of the third series are of the form
    $$
    ht^n\cdot (et^{-1})^\ell\cdot \hat w_\ell= 
    (et^{-1})^\ell ht^n \cdot \hat w_\ell +2
    \ell (et^{-1})^{\ell-1} et^{n-1}\cdot \hat w_\ell,
    $$
    hence they vanish for $n\geq 3$, while for $n=1,2$ they are easily checked to be elements of $a\tWeyl$.
The summands of the first series are given by
    $$
ft^n\cdot (et^{-1})^\ell\cdot \hat w_\ell= 
(et^{-1})^\ell ft^n \cdot \hat w_\ell -
\ell (et^{-1})^{\ell-1} ht^{n-1}\cdot \hat w_\ell 
-\ell (\ell-1) (et^{-1})^{\ell-2} et^{n-2}\cdot \hat w_\ell,
$$
and all terms are zero or in $a\tWeyl$ but for the case $n=1$, for which we get 
\begin{align*}
(et^{-1})& \cdot (ft)\cdot (et^{-1})^\ell\cdot \hat w_\ell =
a^{\ell+1}(et^{-1})^{\ell+1} \cdot  (v_\grl\otimes f^{\ell+1}v_\mu) \\
&- \ell (et^{-1})^{\ell} h \cdot (v_\grl\otimes f^{\ell}v_\mu)
-\ell (\ell-1) (et^{-1})^{\ell} \cdot \hat w_\ell = \hat w_{\ell+1}+ K_1\hat w _\ell
\end{align*}
for some constant $K_1$. Finally,
$e\cdot f\cdot \hat w_\ell+e\cdot f\cdot \hat w_\ell+\frac12h\cdot h \cdot \hat w_\ell$    
belongs to  $K_2\hat w_\ell+a\tWeyl$ for some constant $K_2$. Hence we get 
$$
  S^{(2)}_{1/2} \hat w_\ell \equiv \hat w _{\ell+1}+K \hat w_\ell \bmod a\tWeyl 
$$
for some constant $K$, proving our claim. 
\end{proof}

We now prove that the zero-th semi-infinite cohomology of the module $\tWeyl$ is isomorphic to $Z_2^{\grl,\mu}$.

\begin{theorem}\label{teo:comologiatWeyl2} For $\gog=\gos\gol(2)$ the map $\Phi:Z_2^{\grl,\mu}\lra \semiinfz\big(\tWeyl\big)$ given by  $\Phi(z)= z\cdot [v_\grl\otimes v_\mu]$ is an isomorphism.
\end{theorem}

\begin{proof}
By \cite{FLMM}, Theorem 6.4, the action of $Z_2$ on $\Weyl$, hence on $\tWeyl$, factors through $Z_2^{\grl,\mu}$. Moreover $v_\grl\otimes v_\mu$ is a cycle, so the map $\Phi$ is well defined. 
Since we know that both modules are torsion-free, to prove that $\Phi$  is an isomorphism it suffices to prove that the localization $\Phi_a$ and the specialisation 
$\overline{\Phi}$  are isomorphisms. 

The fact that $\Phi_a$ is an isomorphism is the content of part b) of Proposition \ref{prp:comologiatWeyl1}. 

We need to prove that $\overline{\Phi}$ is an isomorphism. By Lemma \ref{lem:SptWeyl}, Proposition \ref{prp:comologiatWeyl1} and \cite[Theorem 2.13]{FLMM} we have
$$
\frac{Z^{\grl,\mu}_2}{aZ^{\grl,\mu}_2}
\simeq  \prod_{i=0}^\mu Z_1^{\grl+\mu-2i}
\qquad \text{and} \qquad
\frac{\semiinfz(\tWeyl)}{a\semiinfz(\tWeyl)} \simeq 
\bigoplus_{i=0}^{\mu} \semiinfz (\mV_1^{\grl+\mu-2i}).
$$ 
In particular, by Theorem \ref{thm:FG6teo} these two $Z_1$-modules are isomorphic, but we need to prove that our specific map $\overline{\Phi}$ provides an isomorphism between them. By Remark \ref{rmk:Opnudisgiunti} it is enough to prove that $\overline{\Phi}$ is surjective. We prove that the image of $\overline{\Phi}$ contains $\semiinfz (\mF_\ell+a\tWeyl/a\tWeyl)$  arguing by reverse induction on $\ell$. For $\ell=0$ we get our claim. For $\ell>\mu$ there is nothing to prove. Now assume $\ell\leq \mu$. 
Consider again the exact sequence
$$
	\xymatrix@1{
	0\ar[r]
	&\frac{\mF_{\ell+1}+a\tWeyl}{a\tWeyl}\ar[r]
	&\frac{\mF_{\ell}+a\tWeyl}{a\tWeyl}\ar[r]
	& a^\ell U^-_2\otimes_{\mC} G^\ell\ar[r]
	&0.
}
$$ 
We know that the last module is isomorphic to 
\[
a^\ell U^-_2\otimes_{\mC} G^\ell \simeq \mV_1^{\grl+\mu-2\ell}=\Ind_{\hgog_1^+}^{\hgog_1}(V^{\grl+\mu-2\ell})
\]
and that
it is generated by the element $\tilde w_\ell\in a^\ell G^\ell$. Notice this sequence of $Z_1$-modules splits by Remark \ref{rmk:Opnudisgiunti}. Taking semi-infinite cohomology we get a short exact sequence
$$
\xymatrix@1{
	0\ar[r]
	&\semiinfz\left(\frac{\mF_{\ell+1}+a\tWeyl}{a\tWeyl}\right)\ar[r]
	&\semiinfz\left(\frac{\mF_{\ell}+a\tWeyl}{a\tWeyl}\right)\ar[r]
	&\semiinfz\left( a^\ell U^-_2\otimes_{\mC} G^\ell\right)\ar[r]
	&0.
}
$$ 
and we know that the last $Z_2$-module is generated by $\tilde w_\ell$. Hence it is enough to prove that this element is in the image of $Z_{2}^{\grl,\mu}(v_\grl\otimes v_\mu)$ in $\semiinfz\Big(\tWeyl/\mF_{\ell+1}+a\tWeyl\Big)$. By Lemma \ref{lem:calcoloSugawara} we know that $\hat w_\ell$ is in this image. Now we prove that $\tilde w_\ell$ and $\hat w_\ell$  define the same element in the semi-infinite cohomology of $a^\ell U^-_2\otimes_{\mC} G^\ell$. 
 This is  a claim about the cohomology of the module $\mV_1^\nu$ for $\nu=\grl+\mu-2\ell$. For any $\nu$ we prove that 
$\big(et^{-1}\big)^h  v_\nu+\big(et^{-1}\big)^{h-1} v_\nu$ is a coboundary. 
Indeed the boundary operator in the case of $\gos\gol(2)$ is equal to 
$$
d^{(1)}=\psi^*+\sum_{n\in\mZ}(et^n)\otimes \psi^*t^{-1-n},
$$
so a simple computation shows
$$
d^{(1)} \left( \big(et^{-1}\big)^{h-1}  v_\nu\otimes (\psi t^{-1})\statovuotoF \right)=
\big(et^{-1}\big)^{h-1}  v_\nu\otimes \statovuotoF +
\big(et^{-1}\big)^{h}  v_\nu\otimes \statovuotoF,
$$
which implies our claim.
\end{proof}

Recall that in \cite{FLMM} we computed the endomorphism ring of $\Weyl$, showing that it is isomorphic to $Z^{\grl,\mu}_2$. We now prove the same result for the module $\tWeyl$.

\begin{proposition}\label{prp:endomorphismVtilde}
The action of the center $Z_2$ on $\tWeyl$ induces an isomorphism 
$$
Z^{\grl,\mu}_2\simeq \End_{\hgog_2}(\tWeyl).
$$	

\end{proposition}
\begin{proof}
We already recalled at the beginning of the proof of Theorem \ref{teo:comologiatWeyl2} that the action of  $Z_2$ on $\tWeyl$ factors through $Z^{\grl,\mu}_2$. We denote by $\gra:Z_2^{\grl,\mu}\lra \End(\tWeyl)$ this action. Since both modules have no $A$-torsion, in order to prove that $\gra$ is an isomorphism it suffices to show that its localization and 
its specialisation are isomorphisms. Moreover, since our modules are finitely generated and have no torsion we have
\begin{align*}
\End_{\hgog_2}\left(\tWeyl\right)[a^{-1}]&\simeq 
\End_{\hgog_2[a^{-1}]}\left(\tWeyl[a^{-1}]\right) 
\simeq 
\End_{\hgog_{t,s}}\left(\mV^\grl\otimes_Q\mV^\mu_s\right)\\&\simeq Z_t^{\grl}\otimes _Q Z_t^{\mu}\simeq Z_2^{\grl,\mu}[a^{-1}],
\end{align*}
hence the localization of $\gra$ is an isomorphism. 

Finally, we prove that the specialisation of $\alpha$ is also an isomorphism. We have already recalled that by \cite[Theorem 2.13]{FLMM} we have $Z_2^{\grl,\mu}/a Z_2^{\grl,\mu}\simeq \prod_{i=0}^\mu Z_1^{\grl+\mu-2i}$.  Hence
by Theorem \ref{thm:FG6teo} we have the following abstract isomorphisms of $Z_1$-modules:
$$
\frac{Z^{\grl,\mu}_2}{aZ^{\grl,\mu}_2}\simeq \prod_{i=0}^\mu Z_1^{\grl+\mu-2i}\simeq\prod_{i=0}^\mu \End_{\hgog_1}(\mV_1^{\grl+\mu-2i}).
$$ 
Moreover, since $\tWeyl$ has no nontrivial $A$-torsion, by Lemma \ref{lem:SptWeyl} and Remark \ref{rmk:Opnudisgiunti} part (1) we have the inclusion
$$
\frac{\End_{\hgog_1}\left(\tWeyl\right)}{a\End_{\hgog_1}\left(\tWeyl\right)}\subset \End_{\hgog_1}\left(\frac{\tWeyl}{a\tWeyl}\right)\simeq \prod_{i=0}^\mu \End_{\hgog_1}(\mV_1^{\grl+\mu-2i}).
$$
Hence, composing the specialisation of the map $\gra$ with this inclusion and the isomorphisms above we get a $Z_1$-equivariant endomorphism of $\prod_{i=0}^\mu Z_1^{\grl+\mu-2i}$. Moreover, $\gra(1)=1$, hence we conclude by Remark \ref{rmk:Opnudisgiunti} (3) that the specialisation of $\gra$ is also an isomorphism.
\end{proof}


\end{document}